\documentclass[11pt]{amsart}
\usepackage{fullpage}

\usepackage{amsmath}
\usepackage{amssymb}
\usepackage{amsthm}
\usepackage{amssymb,amsmath,amsthm,amscd,mathrsfs,graphicx, color}

\newcommand{\theoremref}[1]{\hyperref[#1]{Theorem~\ref*{#1}}}
\newcommand{\claimref}[1]{\hyperref[#1]{Claim~\ref*{#1}}}
\newcommand{\situationref}[1]{\hyperref[#1]{Situation~\ref*{#1}}}
\newcommand{\lemmaref}[1]{\hyperref[#1]{Lemma~\ref*{#1}}}
\newcommand{\definitionref}[1]{\hyperref[#1]{Definition~\ref*{#1}}}
\newcommand{\propositionref}[1]{\hyperref[#1]{Proposition~\ref*{#1}}}
\newcommand{\conjectureref}[1]{\hyperref[#1]{Conjecture~\ref*{#1}}}
\newcommand{\corollaryref}[1]{\hyperref[#1]{Corollary~\ref*{#1}}}
\newcommand{\exerciseref}[1]{\hyperref[#1]{Exercise~\ref*{#1}}}
\newcommand{\cndref}[1]{\hyperref[#1]{Condition~\ref*{#1}}}
\newcommand{\remref}[1]{\hyperref[#1]{Remark~\ref*{#1}}}

\usepackage{euscript}

\usepackage[all]{xy}

\xyoption{2cell}
\UseAllTwocells
\usepackage{multicol}

\usepackage{xr-hyper}

\usepackage{hyperref}
\usepackage[usenames,dvipsnames]{xcolor}
\hypersetup{colorlinks=true,citecolor=NavyBlue,linkcolor=BrickRed,urlcolor=Orange}

\numberwithin{equation}{section}

\theoremstyle{plain}
\newtheorem{theorem}[equation]{Theorem}
\newtheorem{proposition}[equation]{Proposition}
\newtheorem{lemma}[equation]{Lemma}
\newtheorem{lem}[equation]{Lemma}

\newtheorem{cor}[equation]{Corollary}
\newtheorem{corollary}[equation]{Corollary}

\newtheorem{cnd}[equation]{Condition}

\theoremstyle{definition}
\newtheorem{definition}[equation]{Definition}

\theoremstyle{remark}
\newtheorem{remark}[equation]{Remark}
\newtheorem{rem}[equation]{Remark}

\usepackage{tikz}
\usepackage{tikz-cd}
\usetikzlibrary{matrix,arrows}
\usetikzlibrary{decorations.pathmorphing}

\DeclareMathOperator{\Pic}{Pic}

\DeclareMathOperator{\im}{Im}

\usepackage{tikz}

\def\Pic{\operatorname{Pic}}

\newcommand{\bP}{\mathbb{P}}

\newcommand{\bZ}{\mathbb{Z}}

\newcommand{\calA}{\mathcal{A}}
\newcommand{\calP}{\mathcal{P}}
\newcommand{\calR}{\mathcal{R}}

\newcommand{\calO}{\mathcal{O}}

\newcommand{\calM}{\mathcal{M}}

\begin{document}

\title[]{The moduli space of cubic surface pairs via the intermediate Jacobians of Eckardt cubic threefolds}
\author[]{Sebastian Casalaina-Martin} 
\address{Department of Mathematics, University of Colorado, Boulder, Colorado 80309-0395, USA}
\email{casa@math.colorado.edu}
\author[]{Zheng Zhang} 
\address{Institute of Mathematical Sciences, ShanghaiTech University, Shanghai, 201210 China}
\email{zhangzheng@shanghaitech.edu.cn}

\thanks{The first named author was  partially supported by a grant from  the Simons Foundation (317572).}


\date{\today}

\bibliographystyle{amsalpha}

\begin{abstract}
We study the moduli space of pairs consisting of a smooth cubic surface and a smooth hyperplane section, via a Hodge theoretic period map due to Laza, Pearlstein, and the second named author.  The construction associates to such a pair a so-called Eckardt cubic threefold, admitting an involution, and the period map sends the pair to the anti-invariant part of the intermediate Jacobian of this cubic threefold, with respect to this involution.  Our main result is that the global Torelli theorem holds for this period map; i.e., the period map is injective. To prove the result, we describe the
 anti-invariant part  of the intermediate Jacobian as a Prym variety of a branched cover.  Our proof uses results of Naranjo--Ortega, Bardelli--Ciliberto--Verra, and Nagaraj--Ramanan, on related Prym maps.  In fact, we are able to recover the degree of one of these Prym maps by describing positive dimensional fibers, in the same spirit as a result of Donagi--Smith on the degree of the Prym map for connected \'etale double covers of genus $6$ curves.
\end{abstract}

\maketitle

\section*{Introduction}

Moduli spaces of pairs consisting of a variety, together with a boundary divisor, have become a central focus in moduli theory.  In light of the success of Hodge theory in studying the geometry of moduli spaces, it is natural to construct Hodge theoretic period maps associated with such moduli spaces of pairs.  In this paper, we focus on a special case, namely we consider the moduli space $\mathcal M$ of cubic surface pairs $(S,E)$, where $S$ is a smooth cubic surface, and $E\subset S$ is a smooth hyperplane section.  Moduli spaces of cubic surfaces with boundary divisors have been studied before in a number of contexts, e.g., \cite{HKT_lccpt} and \cite{Friedman_acpair}.  More recently, in the context of log K-stability, some compactifications  of the moduli space $\mathcal M$  were described in \cite{GMGS_logk}, by reducing to the GIT compactifications of cubic surface pairs analyzed in \cite{GMG_cubic2pair}.  Cubic surface pairs $(S,E)$  also provide examples of log Calabi--Yau surfaces, which arise in the study of mirror symmetry  \cite{GHK_logcy}; the mirror family to a particular cubic surface pair was recently given in \cite{GHKS_mirrorcubic}.

In another direction, the Zilber--Pink conjecture \cite[Conj.~1.3]{pink} provides a framework that calls attention to unlikely intersections in Shimura varieties; i.e., subvarieties of Shimura varieties that meet special subvarieties in higher than expected dimension. The moduli space $\mathcal M$ is an unlikely intersection of this type.  More precisely, let $\mathcal C\subset \mathcal A_5$ denote the moduli space of intermediate Jacobians of cubic threefolds sitting inside the moduli space of principally polarized abelian varieties of dimension $5$.  There is a  generically injective finite map $\mathcal M\to \mathcal C$, with image $\mathcal M'$ parameterizing intermediate Jacobians of Eckardt cubic threefolds without a choice of Eckardt point (this is explained below, see also \S\ref{S:Eck=CSpair}); $\mathcal M'$ is an unlikely intersection of $\mathcal C$  with a special subvariety of $\mathcal A_5$ (denoted by $\mathcal A_5^\tau$ in the proof of \propositionref{P:stablequartic}).  Unlikely intersections in the Torelli locus were investigated recently in \cite{moonen_oort}.

In this paper, we study the moduli space $\calM$ of cubic surface pairs $(S, E)$ via a Hodge theoretic period map
constructed  in \cite{LPZ_eckardt}.  More precisely, for our purposes, for the pair $(S,E)$, it is more convenient to consider the associated pair $(S,\Pi)$, where $ \Pi \subset \mathbb P^3$ is the hyperplane such that $E=S\cap \Pi$.  Then, associated with the pair $(S,\Pi)$ is a so-called Eckardt cubic threefold $X \subset \bP^4$, which in coordinates can be described as the zero set of $f(x_0,\dots,x_3)+l(x_0,\dots,x_3)x_4^2$, where $f$ (respectively, ~$l$) is the polynomial of degree $3$ (respectively, ~$1$) on $\mathbb P^3$ defining $S$ (respectively, ~$\Pi$).  Such a cubic threefold $X$ has a natural involution $\tau$, given by $x_4\mapsto -x_4$.  The point $p=[0,0,0,0,1]$ is called an  Eckardt point of $X$.  We then define $JX^{-\tau}$ to be the anti-invariant part of the intermediate Jacobian $JX$ with respect to the induced involution  $\tau$ on $JX$.  The principal polarization on $JX$ induces a polarization on $JX^{-\tau}$ of type $(1,1,1,2)$ so that in the end we obtain a period map: $$\calP: \calM \longrightarrow \calA_4^{(1,1,1,2)}$$
$$
(S,\Pi)\mapsto JX^{-\tau}.
$$

As the moduli space of cubic surface pairs has dimension $7$, while the moduli space of polarized abelian fourfolds has dimension $10$, the period map cannot be dominant.  Our main result is the following global Torelli theorem for  $\calP$.

\begin{theorem}[\theoremref{thm:cubic2pairtorelli},  \corollaryref{C:InfTorHdg}, Torelli theorem for $\calP$] \label{thm:main1}
The period map $\calP: \calM \rightarrow \calA_4^{(1,1,1,2)}$ (which sends $(S, \Pi)$ to the anti-invariant part $JX^{-\tau}$) is injective and has injective differential.
\end{theorem}

We note that several other period maps for cubic hypersurfaces, and cubic hypersurface pairs have been studied recently. 
First, a related period map, for cubic \emph{threefold} pairs, was constructed in \cite{LPZ_eckardt} and \cite{YZ_symcubic} in essentially the same way we constructed the period map $\calP$ above, motivating the construction we use here.  In another direction, Allcock, Carlson and Toledo \cite{ACT_cubic2, ACT_cubic3} studied the moduli space of cubic surfaces (respectively, cubic threefolds) via the period map for cubic threefolds (respectively, cubic fourfolds) by considering the pairs $(\mathbb P^3,S)$ (respectively, $(\mathbb P^4,X)$) where $S$ (respectively, $X$) is a cubic surface (respectively, cubic threefold).

\smallskip  

Turning now to the proof of \theoremref{thm:main1}, it is well-known that two cubic threefolds are isomorphic if and only if their intermediate Jacobians are isomorphic as principally polarized abelian varieties.  This can be proved either by studying the Gauss map on the theta divisor \cite{CG_cubic3} or by analyzing the singularities of the theta divisor \cite{Mumford_prym, Beauville_singtheta}.  
In our situation, the abelian variety is not principally polarized,  so that it is not immediately clear how to formulate a proof of \theoremref{thm:main1} in these terms.  However, a third proof of the Torelli theorem for cubic threefolds can be given by studying the fibers of the Prym map $\calR_6 \rightarrow \calA_5$ over intermediate Jacobians of cubic threefolds: by \cite{DS_prym} (see also \cite{Donagi_prymfiber}) the fiber of the Prym map  over an intermediate Jacobian of a cubic threefold is (a Zariski open subset of) the Fano surface of lines, which determines the cubic threefold up to isomorphism.  We prove \theoremref{thm:main1} in a similar way, via fibers of a Prym map.  We explain this in more detail below.  The proof that the period map has injective differential is proven via a  more direct Hodge theoretic approach (\corollaryref{C:InfTorHdg}).

\smallskip  

Further motivation for our approach to \theoremref{thm:main1} via Prym varieties comes from the fact that realizing $JX^{-\tau}$ as a Prym variety turns out to be important for an ongoing project of Laza, Pearlstein and the second named author, on the LSV construction \cite{LSV_hk} applied to an Eckardt cubic fourfold; recall that the Prym construction of the intermediate Jacobian of a cubic threefold is central to the work in \cite{LSV_hk}. 

In this direction, we prove that $JX^{-\tau}$ is isomorphic to the dual abelian variety of the Prym variety of a double cover of a smooth genus $3$ curve branched at $4$ points.  For the associated Prym map 
$$
\calP_{3,4}:\calR_{3,4} \longrightarrow \calA_4^{(1,2,2,2)},
$$
which is known to be dominant, and generically finite of degree $3$ (\cite{BCV_pav122, NR_pav122}; see also \cite[Thm.~0.3]{NO_prymtorelli}), we show  that the fiber over the dual abelian variety $(JX^{-\tau})^\vee$ is isomorphic to (a Zariski open subset of) the elliptic curve $E = S \cap \Pi$, whose Jacobian is in turn isomorphic to the invariant part $JX^\tau$:
  
\begin{theorem}[\theoremref{thm:prymfiber}, fiber of the Prym map] \label{thm:main2}
Let $X$ be an Eckardt cubic threefold coming from a cubic surface pair $(S, \Pi)$, with involution $\tau$. Consider the proper Prym map 
$$
\calP_{3,4}: \overline{\calR}_{3,4}^a \longrightarrow \calA_4^{(1,2,2,2)}
$$ where $\overline{\calR}_{3,4}^a$ denotes the moduli space of branched allowable double covers  (see \S \ref{S:PrymFiber})
  of genus $3$ curves branched at $4$ points. The fiber of $\calP_{3,4}$ over the dual abelian variety of the anti-invariant part, $(JX^{-\tau})^\vee \in \calA_4^{(1,2,2,2)}$, is isomorphic to the elliptic curve $E = S \cap \Pi$.
\end{theorem}

This is an analogue of the result of Donagi and Smith, mentioned above, that the fiber of the Prym map $\overline{\calR}^a_6 \rightarrow \calA_5$ over the intermediate Jacobian of a smooth cubic threefold is isomorphic to the Fano surface of lines on the cubic threefold \cite{DS_prym, Donagi_prymfiber}.
We prove \theoremref{thm:main2} using a similar technique.  

In fact, Donagi and Smith develop a  technique  to study the degree of a generically finite map by studying the local degree along positive dimensional fibers. Applying this to the case of the Prym map over the intermediate Jacobian locus, they show that the Prym map $\calR_6 \rightarrow \calA_5$ has degree $27$.  Geometrically, this comes down to the fact that a general hyperplane section of a cubic threefold is a smooth cubic surface, which contains $27$ lines.
In a similar way, as a corollary of \theoremref{thm:main2}, we can recover the degree of the Prym map $\calP_{3,4}$.

\begin{corollary}[{\cite{BCV_pav122, NR_pav122},  \cite[Thm.~0.3]{NO_prymtorelli}, degree of the Prym map}] \label{cor:main3}
The degree of the Prym map $\calP_{3,4}: {\calR}_{3,4} \rightarrow \calA_4^{(1,2,2,2)}$ is $3$.  
\end{corollary}

More precisely, we show that the local degree of $\calR_{3,4} \rightarrow \calA_4^{(1,2,2,2)}$ along a fiber $E$ as in \theoremref{thm:main2} is $3$, which implies the degree of the map is $3$ (see \S\ref{S:DS-method}).  The geometric interpretation in our case is that a general hyperplane section of an Eckardt cubic threefold passing through the Eckardt point is an Eckardt cubic surface and hence contains three invariant lines (which also correspond to  the three points of intersection of the elliptic curve $E$ with the hyperplane section).  

\medskip
 
We now give an outline of the paper. 
We work throughout over the complex numbers $\mathbb C$. 
In Section \ref{sec:eckardtcubic3}, we review the construction of Eckardt cubic threefolds $X$ from cubic surface pairs $(S, \Pi)$ following \cite{LPZ_eckardt}.  We consider lines on Eckardt cubic threefolds, for future use with the Prym construction.  We also explain the period map $\mathcal P: \mathcal M\to \mathcal A^{(1,1,1,2)}_4$.
In Section \ref{sec:precubic} we review various realizations of cubic surfaces and cubic threefolds as fibrations in quadrics, as well as the associated  Prym construction due to Mumford showing that the intermediate Jacobian of a cubic threefold $X$ is isomorphic to the Prym variety associated to the double cover of the discriminant curve obtained by projection from a line $\ell\subset X$  (see \cite[App. ~C]{CG_cubic3}) and also \cite{Beauville_schottky, Beauville_prymij}).  Since our cubic threefold $X$ has an involution $\tau$, in order to realize the involution on the intermediate Jacobian via the Prym construction, we must choose a line $\ell\subset X$ that is preserved by the involution.  

It turns out there are two types of $\tau$-invariant lines in $X$: the $27$ lines $\ell$ on the cubic surface $S \subset X$, which are point-wise fixed by $\tau$, and the lines $\ell'$ passing through the Eckardt point $p$, and are not point-wise fixed by $\tau$ (\lemmaref{lem:invline}).  This latter class of lines $\ell'$ form the ruling of the cone $X \cap T_pX$, and are parametrized by the elliptic curve $E$.  In other words, the $\tau$-fixed locus of the Fano surface $F(X)$ of lines in $X$ consists of $27$ isolated points and an elliptic curve isomorphic to $E$. 

The Prym construction then falls into two cases, depending on the type of $\tau$-invariant line chosen.  When projecting  $X$ from one of the $27$ point-wise fixed lines $\ell$ (Section \ref{sec:proj27}), we obtain an \'etale double cover $\widetilde D \rightarrow D$ of a smooth plane quintic $D$ (\propositionref{prop:proj27}).  Moreover, the involution $\tau$ acts naturally on $\widetilde D$, and we obtain a $(\mathbb Z/2\mathbb Z)^2$ automorphism group acting on $\widetilde D$.  Using the ideas of Mumford \cite{Mumford_prym} (see also \cite{Recillas_trigonal, Donagi_prymfiber, RR_prym1, RR_prym2}), one obtains a tower of double covers corresponding to the various subgroups of $(\mathbb Z/2\mathbb Z)^2$ (\lemmaref{lem:kleinD}), and one can identify $JX^{-\tau}$ with the quotient by a nontrivial  $2$-torsion point of the Prym of a double cover of a smooth genus $2$ curve branched at six points (\theoremref{T:27decomp}).  This naturally leads us to the Prym map:
$$
\calP_{2,6}: \calR_{2,6} \rightarrow \calA_4^{(1,1,2,2)}
$$
for connected double covers of smooth curves of genus $2$ branched at $6$  points, 
which was studied by Naranjo and Ortega \cite{NO_prymtorelli, NO_prymtorelli2}, who show that the map is injective.  In fact, a careful analysis using the results in \cite{LO_cyclicprym} on the differential allows us to conclude that our period map $\calP:\calM \rightarrow \calA_4^{(1,1,1,2)}$ is generically finite onto its image. 
As a consequence, one obtains that  $\mathcal P$ has generically  injective differential; 
 we give a direct proof that $\mathcal P$ has injective differential at all points  in \corollaryref{C:InfTorHdg}.

In contrast, when projecting $X$ from a $\tau$-invariant line $\ell'$ through the Eckardt point $p$ (Section \ref{sec:projcone}), we get an allowable double cover of a reducible plane quintic obtained as the union  a plane quartic $C$ and a residual line (\propositionref{prop:projcone}).  Restricting the cover to $C$ one obtains a double cover $\widetilde C \rightarrow C$  of a genus $3$ curve branched at the $4$ points of intersection of the residual line with $C$.  We use  Beauville's theory \cite{Beauville_schottky, Beauville_prymij} to study this cover, and find that $JX^{-\tau}$ is isomorphic to the dual abelian variety of the Prym variety $P(\widetilde C, C)$ (\theoremref{T:conedecomp}), leading us to the Prym map:
$$
\calP_{3,4}: \calR_{3,4} \rightarrow \calA_4^{(1,2,2,2)}
$$ 
for connected double covers of smooth curves of genus $3$ branched at four points.
As described after \theoremref{thm:main2} above, we show that the fiber of $\calP_{3,4}$ over $(JX^{-\tau})^\vee$ is equal to the elliptic curve $E$ parameterizing the $\tau$-invariant lines through $p$.  For this, we review in Section \ref{S:DS} Donagi--Smith's argument for the fiber of the Prym map over the intermediate Jacobian of a cubic threefold.  We then build on this in Section \ref{S:PrymFiber}, to complete the proof of \theoremref{thm:main2}.

In Section \ref{S:GlobaTor}, we complete the proof of \theoremref{thm:main1}.  The basic idea is that given two Eckardt cubics $X_1$ and $X_2$, with $JX_1^{-\tau}\cong JX_2^{-\tau}$ then \theoremref{thm:main2} implies that the two polarized abelian varieties arise  from choosing different lines on the same Eckardt cubic threefold.  The only ambiguity is the Eckardt point. For this, we have \theoremref{T:Eck-from-quart}, showing that one can recover the Eckardt point.

 \smallskip 
We make one final remark on automorphisms.   Automorphisms of prime order of smooth cubic threefolds have been classified in \cite{GL_autocubic}. Let us use the notation in \cite[Thm.~ 3.5]{GL_autocubic}.  Allcock, Carlson and Toledo study cubic threefolds admitting an automorphism of type $T_3^1$ (which is of order $3$) in \cite{ACT_cubic2}.  An isogenous decomposition of the intermediate Jacobian of a cubic threefold admitting an automorphism of type $T_5^1$ (which has order $5$) is given by van Geemen and Yamauchi in \cite{vGY_ijdecomp}.  We study the intermediate Jacobian of an Eckardt cubic threefold (which admits an involution of type $T_2^1$) in this paper.  The study of intermediate Jacobians of cubic threefolds with other types of automorphisms will appear elsewhere.

\subsection*{Acknowledgements} 

We thank heartily Radu Laza and Gregory Pearlstein for many useful discussions related to the subject. 
 We thank Brendan Hassett for mentioning the connection with  unlikely intersections in Shimura varieties. 
The first named author thanks Claire Voisin for a conversation on isogenies of intermediate Jacobians, and Angela Ortega for a conversation on Prym varieties of branched covers, both of which played an important role in our approach to this paper.  The first named author also thanks Jeff Achter for some conversations about Shimura varieties.   The second named author is grateful to Giulia Sacc\`a for several valuable suggestions (in particular, \propositionref{prop:coneprym} is due to her) at the early stage of the project.

\section{Cubic threefolds with an Eckardt point} \label{sec:eckardtcubic3}

\subsection{Definition of an Eckardt cubic threefold} 

We follow \cite[\S 1]{LPZ_eckardt} and briefly recall the construction and some geometric properties of a cubic threefold admitting an Eckardt point. We also fix the notation that will be used  in the remainder of the paper.

\subsubsection{Definition and characterizations of Eckardt cubic threefolds}

\begin{definition}[\cite{LPZ_eckardt} Definition 1.5] \label{def:eckardt}
Let $X$ be a smooth cubic threefold. A point $p$ of $X$ is an \emph{Eckardt point} if $X \cap T_pX$ (where $T_pX$ denotes the projectivized tangent space of $X$ at $p$) is a cone with vertex $p$ over an elliptic curve $E$.
\end{definition}

We call a pair $(X,p)$ consisting of a smooth cubic threefold $X$ with an Eckardt point $p$ an \emph{Eckardt cubic threefold}.  
Eckardt cubic threefolds can be characterized in the following way.  

\begin{proposition} \label{prop:eckardt}
Let $X$ be a smooth cubic threefold. The following statements are equivalent.
\begin{itemize}

\item $X$ admits an Eckardt point $p$.

\item $X$ admits an involution $\tau$ which fixes point-wise a hyperplane section $S\subset X$ and a point $p \in (X-S)$ (which is an Eckardt point).

\item One can choose coordinates on $\bP^4$ such that $X$ has equation 
\begin{equation} \label{eqn:eckardt}
f(x_0, \dots, x_3) + l(x_0, \dots, x_3)x_4^2 = 0,
\end{equation}
(in which case $p=[0,0,0,0,1]$ is an Eckardt point).
\end{itemize}

\end{proposition}
\begin{proof}
See \cite[Lem.~ 1.6, Lem.~ 1.8]{LPZ_eckardt}.
\end{proof}

We now discuss  \propositionref{prop:eckardt} in coordinates. Let $X$ be a cubic threefold in the coordinates of Equation \eqref{eqn:eckardt}. Let $p=[0,0,0,0,1]$. Then we have that 
\begin{equation}\label{E:TpXdef}
T_pX=(l=0)
\end{equation}
 is the projectivized tangent space to $X$ at $p$ and the intersection  $X \cap T_pX$ is a cone with vertex $p$ over the elliptic curve
 \begin{equation}\label{E:Edef}
 E = (f=l=x_4=0).
 \end{equation}
  As a result, $p$ is an Eckardt point of $X$. The involution $\tau$ of \propositionref{prop:eckardt} is given by 
\begin{equation} \label{eqn:tau}
\tau: [x_0, \dots,x_3,x_4] \mapsto [x_0, \dots, x_3, -x_4],
\end{equation}
and the point-wise fixed hyperplane section $S\subset X$ is given by 
\begin{equation}\label{E:Sdef}
S=(f=x_4=0).
\end{equation}

\subsubsection{Equivalence between Eckardt cubic threefolds and cubic surface pairs} \label{S:Eck=CSpair}
We now review the one-to-one correspondence between Eckardt cubic threefolds $(X, p)$ and cubic surface pairs $(S, \Pi)$ where $S$ is a smooth cubic surface and $\Pi$ is a transverse plane in $\bP^3$. 

First, an Eckardt cubic threefold $(X, p)$ determines (up to projective linear transformation) a cubic surface pair $(S, \Pi)$. Namely, the cubic surface $S\subset X$ is the point-wise fixed hyperplane section of \propositionref{prop:eckardt}.  The plane $\Pi$ is easiest to define in coordinates. Starting with an Eckardt cubic threefold $X$ cut out by Equation (\ref{eqn:eckardt}), we have 
\begin{equation} \label{eqn:cubic2pair}
S = (f=x_4=0), \,\,\,\,\,\, \Pi = (l=x_4=0),
\end{equation}
and view $S$ (respectively, $\Pi$) as a cubic surface (respectively, a plane) in $(x_4=0) \cong \bP^3$. The base elliptic curve $E$ of the cone $X \cap T_pX$ is the intersection of $S$ and $\Pi$: 
\begin{equation} \label{eqn:E}
E = (f=l=x_4=0) = S \cap \Pi. 
\end{equation}
Because $X$ is smooth, by \cite[Lem.~ 1.4]{LPZ_eckardt} the cubic surface $S$ is smooth, and $S$ intersects the plane $\Pi$ transversely (which together imply that the elliptic curve $E$ is smooth).

Conversely, a cubic surface pair $(S, \Pi)$ determines an Eckardt cubic threefold $(X,p)$ (up to projective linear transformation). Indeed, we  assume the equation of the cubic surface $S$ (respectively, the plane $\Pi$) in $\bP^3$ is $f(x_0, \dots, x_3)=0$ (respectively, $l(x_0, \dots, x_3)=0$), and consider the cubic threefold $X$ with equation $f(x_0, \dots, x_3) + l(x_0, \dots, x_3)x_4^2=0$.  By rescaling $x_4$, we see that this is independent of the choice of equations of $S$ and $\Pi$.  Since $S$ and $\Pi$ meet transversally, we have by \cite[Lem.~1.3]{LPZ_eckardt} that the cubic threefold $X$ is smooth.

\begin{rem}
 A coordinate free description of $X$ is to take the double cover $Z \rightarrow \bP^3$ branched along the singular quartic $S \cup \Pi$, and then perform some birational modifications to obtain $X$ (explicitly, one blows up $Z$ along the reduced inverse image of $S \cap \Pi$ and then blows down the strict transform of the inverse image of $\Pi$ in $Z$; see \cite[Prop.~ 1.9]{LPZ_eckardt}). A similar correspondence between Eckardt cubic fourfolds and cubic threefold pairs has been used in \cite{LPZ_eckardt} to construct a period map for the moduli space of cubic threefold pairs.
\end{rem}

\subsection{The period map for cubic surface pairs via Eckardt cubics}

In this subsection we define a period map $\calP$ for cubic surface pairs $(S, \Pi)$ using Eckardt cubic threefolds. We shall prove the global Torelli theorem for $\calP$ in Section \ref{S:GlobaTor}. 

Let $X$ be an Eckardt cubic threefold constructed from a cubic surface pair $(S, \Pi)$, as discussed above. By abuse of notation, we use $\tau$ to denote the involution on the principally polarized intermediate Jacobian $JX$ induced by the involution on $X$ in (\ref{eqn:tau}). Define the invariant part $JX^\tau$ and the anti-invariant part $JX^{-\tau}$ respectively by 
\begin{equation} \label{eqn:invantiinv}
JX^\tau = \im (1+\tau), \ \ \  \text{ and } \ \ \ JX^{-\tau} = \im (1-\tau).
\end{equation} 
By \cite[Prop.~ 13.6.1]{BL_cag}, $JX^\tau$ and $JX^{-\tau}$ are $\tau$-stable complementary abelian subvarieties of $JX$. The dimensions of $JX^\tau$ and $JX^{-\tau}$ are computed in the following lemma. 

\begin{lemma} \label{lem:dimpm}
The abelian subvarieties $JX^\tau$ and $JX^{-\tau}$ have dimensions $1$ and $4$ respectively. 
 The principal polarization of $JX$ induces polarizations of type $(2)$ and  $(1,1,1,2)$ on $JX^\tau$ and  $JX^{-\tau}$, respectively.
\end{lemma}

\begin{proof}
One may compute the dimensions of $JX^\tau$ and $JX^{-\tau}$ by computing the dimensions of the positive and negative eigenspaces for the action of $\tau$ on $H^{1,2}(X)$, or $H^{2,1}(X)$, for any  particular Eckardt cubic threefold $X$, such as the one given by  $F=x_0^3+x_1^3+x_2^3+x_3^3+x_0x_4^2=0$.  Identifying the eigenspaces is a standard computation using Griffiths residues (see e.g., \cite[Thm.~3.2.12]{CMSP}).  
The statement on the polarization type follows from \lemmaref{L:Rod_avgroup}, below.
\end{proof}

\begin{lem}[{\cite[Thm.~ 5.3]{Rod_avgroup}}]\label{L:Rod_avgroup}
Let $(A,\Theta)$ be a principally polarized abelian variety, and assume that $\tau:A\to A$ is an involution such that $\tau^*\Theta \equiv_{\operatorname{alg}} \Theta$.  Defining the invariant and anti-invariant sub-abelian varieties 
$$
A^\tau:=\im(1+\tau), \ \ \ \text{ and } \ \ \    A^{-\tau}:=\im(1-\tau),
$$
respectively, we have that $A^\tau$ and $A^{-\tau}$ are complementary sub-abelian varieties of $(A,\Theta)$ in the sense of \cite[p.125, p.365]{BL_cag}, and there is some non-negative integer $r$ such that the polarization types of the restrictions of $\Theta$ to  $A^\tau$ and $A^{-\tau}$ are both of the form
$$
(1,\cdots,1,\underbrace{2,\cdots,2}_r);
$$
(i.e., the number of $1$s  in the polarization type may differ for $A^{\tau}$ and $A^{-\tau}$, but the number of $2$s is the same).
Moreover, the exponent of $A^\tau$ and $A^{-\tau}$ is $2$, unless $r=0$, in which case the exponent is $1$, and then $A=A^\tau\times A^{-\tau}$ as principally polarized abelian varieties.
\end{lem}

\begin{proof}
The assertion on polarizations follows from \cite[Thm.~ 5.3]{Rod_avgroup}.  The entire lemma in fact follows from standard arguments via norm maps for abelian subvarieties of principally polarized abelian varieties (see e.g., \cite[Ch.~12]{BL_cag}).
\end{proof}

\begin{rem}\label{R:dualPol}
Given a polarized abelian variety $(A,\Theta)$ of type $(d_1,\dots,d_g)$, the isogeny $\phi_\Theta:A\to A^\vee$ induces the isogeny  $\phi_\Theta^{-1}:A^\vee\to A$.  One can check that this is induced by a polarization $\Theta^\vee$ on $A^\vee$ of type $(\frac{d_g}{d_g}, \frac{d_g}{d_{g-1}},\dots,\frac{d_g}{d_1})$, which is characterized by the fact that up to translation it is the only ample line bundle on $A^\vee$ such that $\phi_{\Theta^\vee}=\phi_\Theta^{-1}$ and $\phi_\Theta^*\Theta^\vee \equiv_{\text{alg}} d_g\Theta$ (e.g.,  the proof
\footnote{The definition of $\psi_L$ in the first line of the proof of \cite[Prop.~14.4.1]{BL_cag} should be $\psi_L=d_1\phi_L^{-1}$, not $\psi_L=d_1d_g\phi_L^{-1}$.}
 of \cite[Prop.~14.4.1]{BL_cag}, where they take $d_1\phi_\Theta^{-1}$ 
 in place of $\phi_\Theta^{-1}$; see also \cite{BL_isoav}).  In particular, in the language of \lemmaref{L:Rod_avgroup}, we have $A/A^{\tau}\cong (A^{-\tau})^\vee$ and $A/A^{-\tau}\cong (A^{\tau})^\vee$ 
have induced polarization of type $(\underbrace{1,\cdots,1}_r, 2,\dots, 2)$. 
\end{rem}

Let $\calM$ be the moduli space of cubic surface pairs $(S, \Pi)$, where $S$ is a smooth cubic surface and $\Pi$ is a plane in $\bP^3$ meeting $S$ transversely, constructed for instance using GIT \cite{GMG_cubic2pair}. Let $\calA_4^{(1,1,1,2)}$ be the moduli space of abelian fourfolds with a polarization of type $(1,1,1,2)$. Note that $\dim \calM = 4+3=7$ and $\dim \calA_4^{(1,1,1,2)} = \binom{4+1}{2}=10$. We define via \lemmaref{lem:dimpm} a period map:  
\begin{equation}\label{E:Period}
\calP: \calM \longrightarrow \calA_4^{(1,1,1,2)}
\end{equation}
$$
 (S, \Pi) \mapsto JX^{-\tau}$$ 
which sends a cubic pair $(S, \Pi)$ to the anti-invariant part $JX^{-\tau}$ of the intermediate Jacobian of the Eckardt cubic threefold $X$ associated with $(S, \Pi)$.

\subsection{Lines on an Eckardt cubic}

To study the period map $\calP$, we need to analyze how the intermediate Jacobian $JX$ of the Eckardt cubic threefold $(X,p)$ coming from a cubic surface pair $(S, \Pi)$ decomposes with respect to the involution $\tau$. Our strategy will be to project $X$ from a $\tau$-invariant line to exhibit $JX$ as the Prym variety of the associated discriminant double cover. In the following lemma, we classify lines in $X$ that are invariant under the involution $\tau$.  

\begin{lemma} \label{lem:invline}
Let $X$ be a cubic threefold with an Eckardt point $p$ cut out by Equation (\ref{eqn:eckardt}). Let $\ell \subset X$ be a $\tau$-invariant line (i.e., preserved by $\tau$). Then either $\ell$ is a line on the cubic surface $S = (f=x_4=0)$, or $\ell$ passes through the Eckardt point $p=[0,0,0,0,1]$ and is contained in the cone $X \cap T_pX$ over the elliptic curve $E = (f=l=x_4=0)$. 

In other words, there are two types of $\tau$-invariant lines in $X$: the $27$ lines $\ell$ on the cubic surface $S \subset X$ (which are point-wise fixed by $\tau$), and the lines $\ell'$ that  pass through the Eckardt point $p$ (which have only the points  $p$, and $p':=\ell'\cap E$, fixed by $\tau$), and form the ruling of the cone $X \cap T_pX$ over the elliptic curve $E$.
\end{lemma}
\begin{proof}
The fixed locus of the involution $\tau: x_4 \mapsto -x_4$ of $\bP^4$ consists of the Eckardt point $p$ and the hyperplane $(x_4=0)$. If $\tau$ fixes every point of $\ell \subset X$, then $\ell \subset (x_4=0)$ which implies $\ell \subset S$. Otherwise, $\tau$ fixes two points of $\ell \subset X$. One of the points needs to be off the hyperplane $(x_4=0)$ and hence must be the Eckardt point $p$.  Now, since $\ell\subset X$, we have that $\ell\subset T_pX$, so that by definition, 
the other fixed point of $\ell$  is on the elliptic curve $E$.    
\end{proof}

\section{Cubic surfaces and cubic threefolds as fibrations in quadrics} \label{sec:precubic}

In this section, we recall some facts about cubic hypersurfaces and fibrations in quadrics. The key point is to connect constructions that arise in the projection of  a cubic threefold from a line,  to the case of the projection of a cubic surface from a point or line.  This connection between the two cases is central to our analysis of the intermediate Jacobian of an Eckardt cubic threefold. 

\subsection{Cubic surfaces as fibrations in quadrics}

\subsubsection{Projecting a cubic surface from a point} \label{S:CubSurfPoint}

Let $(S,p')$ be a smooth cubic surface $S\subset \mathbb P^3$ together with a point $p'\in S$.  Projecting $S$ from the point $p'$ determines a rational map $S\dashrightarrow \mathbb P^2$, which is resolved by blowing up the point $p'$, yielding a morphism   
$$
\pi_{p'}:\operatorname{Bl}_{p'}S\longrightarrow \mathbb P^2.
$$
The discriminant curve (the locus in $\mathbb P^2$ where the fibers of the projection are singular) is a plane quartic $C$, which is smooth if and only if $p'$ does not lie on a line of $S$, in which case $\pi_{p'}$ is a  fibration in $0$-dimensional quadrics (equivalently, a double cover of $\mathbb P^2$).

Focusing on the case where 
the discriminant curve (equivalently, the branch locus of $\pi_{p'}$) is smooth, it is well-known that the $27$ lines on the cubic surface, together with the exceptional divisor on the blow-up $\operatorname{Bl}_{p'}S$ map isomorphically under $\pi_q$ to the $28$ bitangents (odd theta characteristics) of $C$. 
 In particular, pairs $(S,p')$ consisting of smooth cubic surfaces with a point not contained in a line on the surface, up to projective linear transformations,  are in bijection with pairs $(C,\kappa_C)$, where $C$ is a smooth plane quartic, and $\kappa_C$ is an odd theta characteristic, up to projective linear transformations.  One recovers $(S,q)$ from $(C,\kappa_C)$ by taking the double cover of $\widehat S\to \mathbb P^2$ branched along $C$, and then, viewing $\kappa_C$ as a bitangent to $C$, one blows down the unique $(-1)$-curve in $\widehat S$ that is in the pre-image of the bitangent line, to obtain $S$ (see e.g., \cite[\S 4.2.4, 4.3.7]{Huybrechts_cubic}).

For certain arguments later, it will be convenient to describe this in coordinates.  For this discussion, let $S$ be a smooth cubic surface, and let $p'\in S$ be a point (which for now may be contained in a line on $S$).  
Without loss of generality, we assume that   $p'=[0,0,0,1]$. To project $S$ from $p'$ to $\bP^2_{x_0,x_1,x_2}$, we write the equation of $S$ as
\begin{equation}\label{E:CubSurfPoint}
k(x_0,x_1,x_2)x_3^2 + 2q(x_0,x_1,x_2)x_3 + c(x_0,x_1,x_2)=0
\end{equation}
where $k$, $q$ and $c$ are polynomials of degree $1$, $2$ and $3$ respectively. 
Note that $p'$ is contained in a line of $S$ if and only if  there is a fiber of $\pi_{p'}$ of dimension $1$; i.e., if and only if the locus $(k=q=c=0)$ is nonempty.

Let $N$ be the matrix 
\begin{equation}\label{E:Nsurf}
N =\begin{pmatrix}  
k &  q\\
q &  c\\
\end{pmatrix}.
\end{equation} 
The discriminant curve $C\subset \mathbb P^2$ is then defined by $\det{N}=kc-q^2=0$:
\begin{equation}\label{E:Cdef}
C=(\det{N}=0)=(kc-q^2=0).
\end{equation}
We can now see that $p'$ is not contained in a line of $S$ if and only if the discriminant curve $C$ is smooth.
Indeed, suppose first that  $p'$ is not contained in a line of $S$.  Then $\pi_{p'}$ is a branched double cover, and if $C$ had a node, then the branched double cover, $\operatorname{Bl}_{p'}S$,  would be singular, giving a contradiction.  Conversely, if $p'$ is contained in a line of $S$, then we saw above that $(k=q=c=0)$ is nonempty. But then at a point in that locus, one can easily see that the discriminant $(kc-q^2=0)$ vanishes to order at least $2$.  

We note also that the theta characteristic $\kappa_C$  on $C$ is defined by the  bitangent line $(k=0)$:
\begin{equation}\label{E:kappaCdef}
\kappa_C = \sqrt{(k=0)}
\end{equation}
where by the square root, we mean the unique effective divisor with twice the divisor giving the divisor $(k=0)$.

\begin{remark}
Note that \cite[Prop.~4.2]{Beauville_dethyper} gives another approach to recovering the smooth cubic surface $S$ and point $p'$ from the smooth plane quartic $C$ and the odd theta characteristic $\kappa_C$.   On the one hand, given $(S,p')$ in the coordinates of \eqref{E:CubSurfPoint}, there is a short exact sequence
\begin{equation}\label{E:KappaSurf}
\xymatrix{
0\ar[r]& \mathcal O_{\mathbb P^2}(-2)\oplus \mathcal O_{\mathbb P^2}(-3) \ar[r]^N & \mathcal O_{\mathbb P^2}(-1)\oplus \mathcal O_{\mathbb P^2}  \ar[r]& \kappa_C \ar[r] & 0. 
}
\end{equation}
Conversely, given a smooth plane quartic $C$ and an odd theta characteristic $\kappa_C$, there is a presentation of $\kappa_C$ as in \eqref{E:KappaSurf}, with $N$ as in \eqref{E:Nsurf}.
The cubic surface defined by \eqref{E:CubSurfPoint} using the entries of $N$  is smooth, and together with point $p'=[0,0,0,1]$,  the associated discriminant plane quartic is $C$ and the associated theta characteristic is $\kappa_C$.  
\end{remark}

\subsubsection{Projecting a cubic surface from a line}
Let $(S,\ell)$ be a smooth cubic surface $S\subset \mathbb P^3$ together with a line  $\ell\subset  S$.  Projecting $S$ from the line $\ell$ determines a rational map $\pi_\ell:S\dashrightarrow \mathbb P^1$; as the induced projection $\mathbb P^3\to \mathbb P^1$ is resolved by blowing up the line $\ell$, which is a divisor on $S$,  we have that $\pi_\ell$ is a morphism, yielding a fibration in conics 
$$
\pi_\ell: S\longrightarrow \mathbb P^1.
$$
The discriminant  is a degree $5$ divisor $D\subset \mathbb P^1$. 
 A second degeneracy locus  $Q\subset \mathbb P^1$ can be defined as the degree $2$ divisor in $\mathbb P^1$ over which the associated conics in $S$  meet $\ell$ with multiplicity $2$ at a point (we can put a well-defined scheme structure on $Q$; see the description below in coordinates).    
 
We now describe this in coordinates.  We may  assume $\ell \subset \bP^3$ is cut out by $x_2=x_3=0$, so that the equation of $S$ is of the form
\begin{equation}\label{E:CubSurfLine}
l_1(x_2,x_3)x_0^2+2l_2(x_2,x_3)x_0x_1+l_3(x_2,x_3)x_1^2+
\end{equation}
$$
2q_1(x_2,x_3)x_0+2q_2(x_2,x_3)x_1+c(x_2,x_3)=0
$$
where $l_i$, $q_j$ and $c$ are homogeneous polynomials in $x_2,x_3$ of degree $1$, $2$ and $3$ respectively.   Let $M$ be the matrix
\begin{equation}\label{E:Msurf}
M = \begin{pmatrix}  
l_1 &l_2 &q_1\\
l_2&l_3&q_2\\
q_1&q_2&c\\
\end{pmatrix}.
\end{equation}
The degree $5$ discriminant $D\subset \mathbb P^1_{x_2,x_3}$ for the fibration in quadrics is then defined by the determinant $\det M$:
\begin{equation}\label{E:Dsurf}
D=(\det M=0)
\end{equation}
 and the degree $2$ discriminant divisor $Q$  over which the associated conics in $S$  meet $\ell$ with multiplicity $2$ at a point is given by    
 determinant of the $(3,3)$-minor:
 \begin{equation}\label{E:Msurf33}
M_{3,3} = \begin{pmatrix}  
l_1 &l_2 \\
l_2&l_3\\
\end{pmatrix}
\end{equation}
 \begin{equation}\label{E:Q1}
  Q=(\det M_{3,3}=0)=(l_1l_3-l_2^2=0).
\end{equation}

\subsubsection{Projecting a cubic surface with an elliptic curve from a line}
Suppose now we have a triple $(S,E,\ell)$ consisting of a cubic surface $S$, a smooth hyperplane section $E\subset S$ and a line $\ell\subset S$.  In this case we obtain another discriminant divisor $B\subset \mathbb P^1$, namely the degree $4$ branch divisor of the the restriction $\pi_\ell|_E:E\to \mathbb P^1$.  Note that since $E$ is smooth, the branch locus $B$ consists of $4$ distinct points.

In coordinates, again, 
we may  assume $\ell \subset \bP^3$ is cut out by $x_2=x_3=0$.  Moreover, since $E$ is irreducible (in fact smooth), it does not contain $\ell$, so the linear equation $(l(x_0,x_1,x_2,x_3)=0)$  cutting $E$ on $S$ must contain $x_0$ or $x_1$ with non-zero coefficient.  Interchanging $x_0$ and $x_1$, we may assume it is $x_0$.  Then after a change of coordinates, $x_0\mapsto l$, $x_i\mapsto x_i$, $i=1,2,3$, we may assume that $l=x_0$.  Thus the equation for $S$ in these coordinates is given by \eqref{E:CubSurfLine}, and the equation for $E$ on $S$ is given by $(x_0=0)$:
\begin{equation}
E=(x_0=0)\cap S.
\end{equation}

Consequently,  the branch divisor $B$ is obtained  by setting $x_0=0$ in \eqref{E:CubSurfLine}, and considering the discriminant, it is  therefore given by the determinant of the minor: 
\begin{equation}\label{E:Msurf11}
M_{1,1} = \begin{pmatrix}  
l_3&q_2\\
q_2&c\\
\end{pmatrix}
\end{equation}
\begin{equation}\label{E:B}
B=(\det M_{1,1}=0)=(l_3c-q_2^2=0). 
\end{equation}

\begin{cnd}[Generic Eckardt cubic threefold] \label{cond:generic}  
Let $(X,p,\ell)$ be a triple consisting of an Eckart cubic threefold $(X,p)$ with associated cubic surface pair $(S,E)$, together with a line $\ell \subset S$.    The associated discriminant divisors $Q$ \eqref{E:Q1} and $B$ \eqref{E:B} associated to the triple $(S,E,\ell)$ have disjoint support.  
\end{cnd} 

 We say that an Eckardt cubic $(X,p)$  satisfies \cndref{cond:generic}, if  there exists a line $\ell\subset S$ such that $(X,p)$ and $\ell$ satisfy \cndref{cond:generic}.
 
 \begin{remark}\label{R:condGen}
We see, for instance from \eqref{E:Q1} and \eqref{E:B}, that the locus of    Eckardt cubics threefolds $(X,p)$ that  satisfy \cndref{cond:generic} is open in moduli.
\end{remark}

\begin{rem}
Although we do not need this, 
one can show that if $(X,p,\ell)$ satisfies \cndref{cond:generic}, then the supports of $Q$ and $B$ are both reduced.
Moreover, one can show that if $(X,p)$ satisfies \cndref{cond:generic}, then for every $\ell\subset S$, we have that $(X,p,\ell)$ satisfies \cndref{cond:generic}.  
\end{rem}

\subsection{Cubic threefolds as fibrations in conics}\label{S:CubSolid}

Let $(X,\ell)$ be a cubic threefold $X\subset \mathbb P^4$ together with a line   $\ell\subset  X$.  
Projecting $X$ from the line $\ell$ determines a rational map $X\dashrightarrow \mathbb P^2$, which is resolved by blowing up the line $\ell$, yielding a fibration in conics:
\begin{equation}\label{E:pi-ell-solid}
\pi_\ell:\operatorname{Bl}_\ell X\longrightarrow \mathbb P^2.
\end{equation}
The discriminant  is a plane quintic $D\subset \mathbb P^2$ with at worst nodal singularities, and there is an associated pseudo-double cover  $\widetilde D\to D$ 
determined by interchanging the lines in the fiber of $\pi_\ell$ over the points of $D$.

Recall that a  double cover of stable curves $\widetilde D\to D$ with associated involution $\iota:\widetilde D\to \widetilde D$ is called \emph{admissible} if the only fixed points of the involution are nodes, and at each node of $\widetilde D$ fixed by $\iota$, the local branches of $\widetilde D$ are not interchanged by $\iota$.
An  admissible double cover is said to be \emph{allowable}
(\cite{DS_prym, Donagi_prymfiber}, \cite[p.173, (**)]{Beauville_schottky})
 if the associated Prym is compact, and an allowable double cover is said to be a \emph{pseudo-double cover} (\cite[Def.~0.3.1]{Beauville_prymij}, \cite[p.157, (*)]{Beauville_schottky})  if, moreover, the fixed points of the associated involution $\iota$ on $\widetilde D$ are exactly the singular points of $\widetilde D$.

For general $\ell$, one has that $D$ is smooth and $\widetilde D\to D$ is \'etale. Associated to the double cover $\widetilde D\to D$ is a  rank-$1$ torsion-free sheaf  $\eta_D$ on $D$ with $\mathcal Hom(\eta_D,\mathcal O_D)\cong \mathcal \eta_D$, and therefore a theta characteristic $\kappa_D=\eta_D\otimes \mathcal O_D(1)$ ($\mathcal Hom (\kappa_D, \omega_D)\cong \kappa_D$), which is odd ($h^0(D,\kappa_D)=1$). One has that $\eta_D$ (and therefore $\kappa_D$) is a line bundle at a point $d\in D$ if and only if $\widetilde D \to D$ is \'etale over $d$. In particular, pairs $(X,\ell)$ consisting of cubic threefold with a line, up to projective linear transformations,  are in bijection with pairs $(D,\kappa_D)$, where $D$ is a stable plane quintic, and $\kappa_D$ is an odd theta characteristic, up to projective linear transformations.  These results are all due to Beauville \cite{Beauville_prymij}; we refer the reader to \cite{CMF} where this is discussed further.  In particular, we are using \cite[Thm.~4.1, Prop.~4.2]{CMF}.
 We will explain this in more detail below, where we express everything in coordinates. 
 A second degeneracy locus  $Q\subset \mathbb P^2$ can be defined as the degree $2$ divisor in $\mathbb P^2$ over which the associated conics in $X$  meet $\ell$ with multiplicity $2$ at a point (we can put a well-defined scheme structure on $Q$; see the description below in coordinates).    
  
We now describe this in coordinates.  We may  assume $\ell \subset \bP^4$ is cut out by $x_2=x_3=x_4=0$, so that the equation of $X$ is of the form
\begin{equation}\label{E:CubSolid}
l_1(x_2,x_3,x_4)x_0^2+2l_2(x_2,x_3,x_4)x_0x_1+l_3(x_2,x_3,x_4)x_1^2+
\end{equation}
$$
2q_1(x_2,x_3,x_4)x_0+2q_2(x_2,x_3,x_4)x_1+c(x_2,x_3,x_4)=0
$$
where $l_i$, $q_j$ and $c$ are homogeneous polynomials in $x_2,x_3,x_4$ of degree $1$, $2$ and $3$ respectively.  We note the similarity with Equation \eqref{E:CubSurfLine}; namely, that equation is given by setting $x_4=0$ in \eqref{E:CubSolid} above.  Geometrically, this corresponds to restricting the fibration in conics here to the line $(x_4=0)\subset \mathbb P^2_{x_2,x_3,x_4}$, or, equivalently, to projecting the cubic surface $S=X\cap (x_4=0)$ from the point $p'=\ell\cap S$.   We will use similar notation to facilitate the translation from one case to the other.

Let $M$ be the matrix
\begin{equation}\label{E:Msolid}
M = \begin{pmatrix}  
l_1 &l_2 &q_1\\
l_2&l_3&q_2\\
q_1&q_2&c\\
\end{pmatrix}.
\end{equation}
The degree $5$ discriminant $D\subset \mathbb P^2$ for the fibration in conics is then defined by the determinant $\det M$:
\begin{equation}\label{E:Deqsolid}
D=(\det M=0).
\end{equation}

The fact that $D$ has at worst nodes follows from the fact that $X$ is smooth (e.g., \cite[Prop.~1.2(iii)]{Beauville_prymij}), and the fact that $D$ is stable can then be deduced from  Bezout's theorem. 
In fact, since the total space $\operatorname{Bl}_\ell X$ is smooth, we have that at any point of $\mathbb P^2$, the corank of $M$ is at most $2$, and  $d\in D$ is a smooth point of $D$ if and only if  the corank of $M$ at $d$  is equal to $1$.  

The fact that $\widetilde D\to D$ is a pseudo-double cover is shown in \cite[Prop.~1.5]{Beauville_prymij}, and it is shown in  \cite{CG_cubic3} (see also \cite[Lem.~1.5]{Murre_cubic3})  
that for a general line, i.e., outside of a divisor in the Fano surface of lines on $X$,   the discriminant  $D$ is smooth and the double cover is connected and \'etale.

Moreover, Beauville has shown there is a short exact sequence (see \cite[Thm.~4.1]{CMF})
\begin{equation}\label{E:KappaSolid}
\xymatrix{
0\ar[r]& \mathcal O_{\mathbb P^2}(-2)^{\oplus 2}\oplus \mathcal O_{\mathbb P^2}(-3) \ar[r]^<>(0.5)M & \mathcal O_{\mathbb P^2}(-1)^{\oplus 2}\oplus \mathcal O_{\mathbb P^2}  \ar[r]& \kappa_D \ar[r] & 0. 
}
\end{equation}
Conversely, given a pair $(D,\kappa_D)$ with $D$ a stable plane quintic and $\eta_D$ an odd theta characteristic, one obtains a presentation  of $\kappa_D$ as in \eqref{E:KappaSolid}, for a matrix $M$ as in \eqref{E:Msolid}.   The cubic threefold $X$ defined by Equation \eqref{E:CubSolid} using the entries of $M$ is smooth, and the projection from the line $\ell=(x_2=x_3=x_4=0)$ gives the discriminant $D$ and odd theta characteristic $\eta_D$ up to projective linear transformations (see \cite[Prop.~4.2]{CMF}). 

The degree $2$ discriminant divisor $Q$  over which the associated conics in $S$  meet $\ell$ with multiplicity $2$ at a point is given by    the 
 determinant of the $(3,3)$-minor:
 \begin{equation}\label{E:M3fold33}
M_{3,3} = \begin{pmatrix}  
l_1 &l_2 \\
l_2&l_3\\
\end{pmatrix}
\end{equation}
 \begin{equation}\label{E:Q2}
  Q=(\det M_{3,3}=0)=(l_1l_3-l_2^2=0).
\end{equation}

The odd theta characteristic $\eta_D$ has yet another description, in terms of $Q$.  Assuming that $Q\cap D$ meet at smooth points of $Q$ and $D$, then one has that the intersection multiplicity at each point of the intersection is even, and we have that
$$
\kappa_D=\sqrt{(\det M_{3,3}=0)}=\sqrt{(l_1l_3-l_2^2=0)}
$$
where by the square root we mean the unique effective divisor such that twice the divisor is the divisor $(l_1l_3-l_2^2)=0$ on $D$.  This is explained in the proof \cite[Prop.~4.2]{CMF}, where in this part of the proof it is assumed that $Q$ and $D$ are smooth; however, the computation is local, and requires only that $Q$ and $D$ meet at smooth points. 

Finally we recall \cite[Thm.~2.1(iii)]{Beauville_prymij}: there is a canonical isomorphism of principally polarized abelian varieties 
\begin{equation}\label{E:P=IJ}
P(\widetilde D,D)\cong JX.
\end{equation}

\section{Eckardt cubic threefolds as fibrations in conics 1: lines through the Eckardt point} \label{sec:projcone}

Let $X$ be a cubic threefold with an Eckardt point $p$ cut out by Equation (\ref{eqn:eckardt}). As in Section \ref{sec:eckardtcubic3}, we denote by $\tau$ the involution associated with the Eckardt point $p$ (see (\ref{eqn:tau})) and set $(S, \Pi)$ to be the corresponding cubic surface pair (see (\ref{eqn:cubic2pair})). In this section, we give another description of the anti-invariant part $JX^{-\tau}$  of the intermediate Jacobian $JX$ (see (\ref{eqn:invantiinv})) by projecting $X$ from a $\tau$-invariant line $\ell' \subset X$ passing through the Eckardt point $p$ (note that the line $\ell'$ belongs to the cone $X \cap T_pX$ over the elliptic curve $E = S \cap \Pi$; see \lemmaref{lem:invline}). 

We now revisit \S \ref{S:CubSolid} in the case of an Eckardt cubic threefold and a line through the Eckardt point.  More precisely, 
let $(X,p)$ be an Eckardt cubic threefold, and let $\ell' \subset X$ be a line passing through the Eckardt point $p$.
Let us choose coordinates so that   $X$ is given by \eqref{eqn:eckardt} $f(x_0,\dots,x_3)+l(x_0,\dots,x_3)x_4^2=0$, with Eckardt point
$$
p=[0,0,0,0,1].
$$   
Without loss of generality, we assume that $\ell'$ intersects the plane $(x_4=0)$ at the point 
$$
p'=[0,0,0,1,0].
$$ 
Because $\ell'$ passes through the Eckardt point $p=[0,0,0,0,1]$, the equation of $\ell'$ is 
$$
\ell'=(x_0=x_1=x_2=0).
$$

We may now write the equation of $X$ in the form $\alpha x_3^3+k(x_0,x_1,x_2)x_3^2 + 2q(x_0,x_1,x_2)x_3 + c(x_0,x_1,x_2) + l(x_0,x_1,x_2,x_3)x_4^2=0$, where $k$, $q$ and $c$ are polynomials of degree $1$, $2$ and $3$ respectively.  Since $X$ contains the line $\ell'=\{[0,0,0,x_3,x_4]\}$, one finds that $\alpha=0$, and the coefficient of $x_3$ in $l(x_0,\dots,x_3)$ is zero.
 
Thus we may write the equation of $X$ as 
\begin{equation} \label{eqn:X'}
\underbrace{k(x_0,x_1,x_2)x_3^2 + 2q(x_0,x_1,x_2)x_3 + c(x_0,x_1,x_2)}_{f(x_0,\dots,x_3)} + l(x_0,x_1,x_2)x_4^2=0
\end{equation} 
The cubic surface $S$ is given by 
$$
S=(f=x_4=0)
$$
and the hyperplane section $E\subset S$ is given by 
$$
E=(l=0)\cap S=(l=f=x_4=0).
$$

We now project $X$ from the line $\ell'$ to the complementary plane $\mathbb P^2_{x_0,x_1,x_2}=(x_3=x_4=0)$, and obtain a conic bundle $\pi_{\ell'}: \operatorname{Bl}_{\ell'}X\to \mathbb P^2_{x_0,x_1,x_2}$.  Considering the form of \eqref{eqn:X'} and \eqref{E:CubSolid},  
we find that the associated matrix $M$ \eqref{E:Msolid} 
is 
$$
M=\begin{pmatrix}  
k & 0 & q\\
0 & l & 0\\
q & 0 & c\\
\end{pmatrix}
$$
and therefore, from \eqref{E:Deqsolid}, the discriminant plane quintic $D$ of the conic bundle $\operatorname{Bl}_{\ell'}X \rightarrow \bP^2$ is given by 
\begin{equation} \label{eqn:CL}
\det \begin{pmatrix}  
k & 0 & q\\
0 & l & 0\\
q & 0 & c\\
\end{pmatrix} = l(kc-q^2) = 0.
\end{equation}

Clearly, the (nodal) discriminant plane quintic $D\subset \mathbb P^2$  consists of two components: 
 a line $L$
\begin{equation}\label{E:Ldef}
L=(l=0)
\end{equation}
and
a (possibly reducible) plane quartic  $C$ 
\begin{equation}\label{E:pq-EckCub}
C=(kc-q^2=0).
\end{equation}

Geometrically, the plane quartic  $C$ is the discriminant for the projection of the cubic surface $S$ from the point $p'\in S$. 
In fact, given the form of $f$ in \eqref{eqn:X'}, and comparing with \eqref{E:CubSurfPoint}, then for the map $\pi_{p'}:\operatorname{Bl}_{p'}S\to \mathbb P^2$, the associated matrix $N$ \eqref{E:Nsurf} is given as
\begin{equation*}
N =\begin{pmatrix}  
k &  q\\
q &  c\\
\end{pmatrix}
\end{equation*} 
so that following \eqref{E:Cdef}, the discriminant plane quartic for $\operatorname{Bl}_{p'}S\to \mathbb P^2$ is given as in \eqref{E:pq-EckCub}.  Recall that the odd theta characteristic $\kappa_C$ determining the pair $(S,p')$, described in  \eqref{E:kappaCdef}, is given by 
\begin{equation}\label{E:kappCEck}
\kappa_C=\sqrt{(k=0)}
\end{equation}
i.e., the pair $(S,p')$  is determined by the bitangent to $C$ cut by $(k=0)$.  

\begin{proposition}[Projecting from a line with $p\in \ell'\subset X$] \label{prop:projcone}
 Let $(X,p)$ be an Eckardt cubic threefold, and let $\ell'\subset X$ be a line passing through $p$.  The following are equivalent:
\begin{enumerate}
\item The line $\ell'$ does not intersect any lines on $S$. (From \lemmaref{lem:invline}, the lines through $p$ are parameterized by their intersection with $E\subset S$, so the general line through $p$ satisfies this condition.)
\item The discriminant plane quintic $D$ is a union $D=C\cup L$ of a smooth plane quartic $C$ and a transverse line $L$, and the associated double cover $\widetilde D\to D$ is a pseudo-double cover. 
\end{enumerate}
\end{proposition}
\begin{proof}
As explained in \S \ref{S:CubSolid}, the cover $\widetilde D\to D$ is always a pseudo-double cover, and as explained above, in our situation,  the discriminant $D$ is always the union of a possibly reducible nodal plane quartic with a transverse line. In other words,  the content of the proposition is to establish exactly  when the plane quartic $C$ is smooth.  From the discussion above, the plane quartic $C$ is the discriminant for the projection of $S$ from $p'$.  We saw in \S \ref{S:CubSurfPoint} that the discriminant $C$ is smooth if and only if $p'$ does not lie on a line of $S$.  
\end{proof}

\begin{lem}\label{L:kappaCEck} 
Let $(X,p)$ be an Eckardt cubic threefold, and let $\ell'\subset X$ be a line passing through $p$. We project $X$ from $\ell'$ as in \propositionref{prop:projcone}, and suppose that the quartic component $C$ is smooth. Let $\widetilde C \to C$ be the associated double cover coming from the conic bundle $\operatorname{Bl}_{\ell'}X\to \mathbb P^2$ (in other words, we restrict the double cover $\widetilde D\to D$ in \propositionref{prop:projcone} to $C$). The double cover $\widetilde C\to C$ determines the odd theta characteristic $\kappa_C$, and conversely. 
\end{lem}

\begin{proof}
The data of the double cover $\widetilde C\to C$ can be described equivalently as a triple
$(C,\beta,\mathcal L,s)$, where  $\beta= C \cap L$ is the branch divisor on $C$, $\mathcal L$ is a line bundle on $C$ with an isomorphism  $\mathcal L^{\otimes 2}\cong \mathcal O_C(\beta)$, and $s$ is section of $\mathcal O_C(\beta)$ vanishing on $\beta$.    The claim is that $\mathcal L\cong \kappa_C$; i.e., that $(C,\beta, \mathcal L,s)$ gives the same cover as $(C,\beta,\kappa_C,s)$.  Indeed, the cover associated to $(C,\beta,\kappa_C,s)$ has function field $K(C)(\sqrt{k})$ \eqref{E:kappCEck}; here we are working on the open subset of $C$ given by $(l\ne 0)$.  At the same time, \cite[Lem.~1.6]{Beauville_prymij} states that the cover of the full discriminant $\widetilde D\to D$, when restricted to the open subset of $D$ given by $(l\ne 0)$ (which is exactly the cover $\widetilde C\to C$ restricted to the open subset $(l\ne 0)$) has function field $K(C)(\sqrt{k})$; i.e., considering \eqref{eqn:CL}, the minor $\det M_{3,3}=kl$, and we can choose coordinates on the plane so that local affine coordinates are given by setting $l=1$.\end{proof}

Just as Beauville (see \S \ref{S:CubSolid}) has shown that the data of a pair $(X,\ell)$ consisting of a smooth cubic threefold $X$ and a line $\ell\subset X$ is equivalent to the data $(D,\eta_D)$ consisting of a nodal plane quintic $D$ and an odd theta characteristic $\kappa_D$, we have a similar result for Eckardt cubic threefolds. 

\begin{theorem}[Reconstructing an Eckardt cubic from a plane quartic]\label{T:Eck-from-quart}
Given a triple $(C,\kappa_C,L)$ consisting of a smooth plane quartic, an odd theta characteristic (bitangent) $\kappa_C$ on $C$,  and a line $L$ meeting $C$ transversely, one can associate a triple $(X,p,\ell')$ consisting of an Eckardt cubic threefold and a line $\ell'$ on $X$ through the Eckardt point $p$ and not meeting any lines on the cubic surface $S\subset X$,  such that $(C,\kappa_C,L)$ are associated to projection from $\ell'$.    In other words, the data of a triple $(X,p,\ell')$ is equivalent to the data $(C,\kappa_C,L)$.  More precisely, 
 a  triple $(X\subset \mathbb P^4,p,\ell')$ recording the embedding of $X$ in $\mathbb P^4$ determines 
a triple $$(C\subset \mathbb P^2,
\xymatrix@C=1em{
0\ar[r]& \mathcal O_{\mathbb P^2}(-2)\oplus \mathcal O_{\mathbb P^2}(-3) \ar[r]^<>(0.5)N & \mathcal O_{\mathbb P^2}(-1)\oplus \mathcal O_{\mathbb P^2}  \ar[r]& \kappa_C \ar[r] & 0 
}, L\subset \mathbb P^2)$$ recording the embedding of $C\cup L$ in $\mathbb P^2$, and a presentation of $\kappa_C$, and conversely.  
\end{theorem}

\begin{proof} Assume we are given $(C,\kappa_C,L)$ as in the theorem.  
Recall from \S \ref{S:CubSurfPoint} that the pair $(C,\kappa_C)$ determines a pair $(S,p')$ consisting of a smooth cubic surface $S\subset \mathbb P^2$, and a point $p'$ not lying on any lines in $S$.  Let $f(x_0,x_1,x_2,x_3)=0$ be the equation of $S$, and let $l(x_0,x_1,x_2)=0$ be the equation of $L$.  As explained in \S \ref{S:CubSurfPoint}, we way take coordinates on $\mathbb P^3$ such that $f=k(x_0,x_1,x_2)x_3^2 + 2q(x_0,x_1,x_2)x_3 + c(x_0,x_1,x_2)$.  Now consider the cubic hypersurface $X\subset \mathbb P^4$ defined by the equation \eqref{eqn:X'}; set $\ell'=(x_0=x_1=x_2=0)$ and  $p=[0,0,0,0,1]$.  Assuming $X$ is smooth,
 it is an Eckardt cubic threefold with Eckardt point $p$, since it admits the involution $x_4\to -x_4$ fixing the hyperplane section $S\subset X$ and the point $p\in X-S$ (\propositionref{prop:eckardt}).   Moreover, projection from $\ell'$ gives the triple $(C,\kappa_C,L)$.  

We now show that $X$ is smooth. This will follow from the fact that $L$ meets $C$ transversally.  Let $\Pi\subset \mathbb P^3$ be defined by $(l=0)$.  We have seen that to show $X$ is smooth it suffices to show that $S$ and $\Pi$ meet transversally (\S \ref{S:Eck=CSpair}). Note that $p'\in \Pi$.   To investigate the intersection $S\cap \Pi$, let $C'$ be the closure of the set of points $s$ in $S-p'$ such that the line $\overline{p's}$ meets $S$ tangentially.   For the fibration in quadrics $\operatorname{Bl}_{p'}S\to \mathbb P^2$, one can check that the proper transform of $C'$ is isomorphic to $C$ (this is a special case of a more general statement about smooth fibrations in quadrics where the corank of the fibers is at most $1$).  Now suppose that $\Pi$ does not meet $S$ transversally at some point $s\in S-p'$.
Then since $p'\in \Pi$, we would have that $s\in C'$, and it would follow that $\Pi$ was the tangent plane to $S$ at this point.  But then the Zariski tangent space to $C'\subset S$ at this point would be contained in $\Pi$, so that projecting from $p'$, we would have the tangent space to $C$ being  contained in $L$, which we have assumed is not true.  
Next let us rule out $\Pi$ meeting $S$ tangentially at $p'$; i.e., $\Pi=T_{p'}S$.  In this case, the intersection $E:=S\cap \Pi$ would be  a singular plane cubic; since we have assumed that $p'$ does not lie on a line of $S$, then $E$ would have to be an irreducible nodal plane cubic, with node at $p'$.  But then the two lines in $\Pi$ forming the tangent cone $C_{p'}E$ would show that $C'$ passed through $p'$.   Consequently, again, we would have that  the Zariski tangent space to $C'\subset S$ at $p'$  would be contained in $\Pi$, so that projecting from $p'$, we would have the tangent space to $C$ being  contained in $L$, which we have assumed is not true. 
 
Thus $S$ and $\Pi$ meet transversally, and so $X$ is smooth, completing the proof. 
\end{proof}

Let us now assume that $\ell' \subset X$ does not meet any lines on $S$, and let us describe the pseudo-double cover $\pi:\widetilde D\to D$ in more detail. First, \propositionref{prop:projcone} shows that $D=L\cup C$, and therefore since $\pi$ is a pseudo-double cover it follows that it can be described as
$\pi:\widetilde D=\widetilde L\cup \widetilde C \longrightarrow L\cup C=D$,
where $\widetilde L\to L$ is a connected double cover of the line $L$ branched at the four points $L\cap C$, while  $\widetilde C\to C$ is a connected double cover of the smooth plane quartic $C$ branched also at the four points $C\cap L$.  
Considering the equation \eqref{eqn:X'} for $X$, and the fact that $L$ is defined by $(l=0)$ \eqref{E:Ldef}, we see that the two lines in $X$ lying over a point in $L$ via the projection $\pi_{\ell'}$ are parameterized by the two points in $S$ lying over the point in $\ell$ via the projection $\pi_{p'}$.  In other words, all together we obtain the intersection $S\cap \Pi=E$; i.e., $\widetilde L\cong E$. Thus the pseudo-double cover can be described as:

\begin{equation}\label{E:CcupLcover}
\pi:\widetilde D=E\cup \widetilde C \longrightarrow L\cup C=D.
\end{equation}

We now give an explicit description of the Prym variety $P(E \cup \widetilde C, L \cup C)$ following \cite[\S 0.3]{Beauville_prymij}.
Let $\widetilde{\nu}:  E \amalg \widetilde{C}  \rightarrow E \cup \widetilde{C}$ and $\nu: L \amalg C \rightarrow L \cup C$ be the normalizations of $E\cup \widetilde{C} $ and $L\cup C$ respectively. Let $\pi': E \amalg \widetilde{C}   \rightarrow L \amalg C$ be the double covering map induced by $\pi$. Denote the ramification points of $\widetilde{C} \rightarrow C$ (respectively, $E \rightarrow L$) by $\widetilde{c}_1, \widetilde{c}_2, \widetilde{c}_3, \widetilde{c}_4 \in \widetilde{C}$ (respectively, $e_1, e_2, e_3, e_4 \in E$).  Note that $\widetilde{\nu}(\widetilde{c}_i) = \widetilde{\nu}(e_i)$ for $1 \leq i \leq 4$. Set $H'$ to be the subgroup of $\Pic(\widetilde{C} \amalg E)$ generated by $\calO(\widetilde{c}_i-e_i)$ ($1 \leq i \leq 4$). Let $H$ be the image of $H_0:=H' \cap J(\widetilde{C} \amalg E)$ in the quotient abelian variety $J( E\amalg \widetilde{C})/\pi'^*J(L \amalg C) \cong E \times (J(\widetilde{C})/\pi^*J(C))$. By \cite[Exer.~ 0.3.5]{Beauville_prymij}, 
the group $G$ consists of $2$-torsion elements and is isomorphic to $(\bZ/2\bZ)^2$. More explicitly, 
$$H =\{0, (\calO_E(e_2-e_1), \calO_{\widetilde{C}}(\widetilde{c}_2-\widetilde{c}_1)), (\calO_E(e_3-e_1), \calO_{\widetilde{C}}(\widetilde{c}_3-\widetilde{c}_1)), (\calO_E(e_4-e_1), \calO_{\widetilde{C}}(\widetilde{c}_4-\widetilde{c}_1))\}.
$$

Recall that $J(\widetilde{C})/\pi^*J(C)$ is the dual abelian variety to $P(\widetilde C,C)$ 
and therefore comes with a dual polarization (Remark \ref{R:dualPol}).  We equip $E$ with its canonical principal polarization.   Note also that as in \eqref{E:sigmaDef}, the involution $\tau$ on $X$ induces an involution $\sigma$ on $\widetilde D$. Now we have the following proposition whose proof we learned from Giulia Sacc\`a.

\begin{proposition} \label{prop:coneprym}
Let $(P(\widetilde D,D),\Xi)$ be the principally polarized Prym variety.  
There is an isogeny of polarized abelian varieties
$$
\phi:E \times (J(\widetilde{C})/\pi^*J(C))\longrightarrow P(E\cup \widetilde C,L\cup C)=P(\widetilde D,D)
$$
with kernel $H\cong(\mathbb Z/2\mathbb Z)^2$.  

Moreover, with respect to the action of $\sigma$ on $(P(\widetilde D,D),\Xi)$, the isogeny $\phi$ induces isomorphisms of polarized abelian varieties $P(\widetilde D,D)^\sigma\cong E$ and $P(\widetilde D,D)^{-\sigma}\cong J(\widetilde{C})/\pi^*J(C)$.  
\end{proposition}

\begin{proof}
The existence of the isogeny $\phi$ and the description of the kernel is  \cite[Prop.~ 0.3.3]{Beauville_prymij} 
We now explain the assertion regarding $P(\widetilde D,D)^\sigma$ and $P(\widetilde D,D)^{-\sigma}$.  
Let $\iota$ be the covering involution associated with the double cover $\pi: E \cup \widetilde{C} \rightarrow L \cup C$. 
We claim that the action of $\sigma$ on $E$ is trivial, while the action of $\sigma$ on $\widetilde C$ coincides with $\iota$.  
Recall that $E \cup \widetilde{C}$ parameterizes the residual lines to $\ell'$ in a degenerate fiber of the conic bundle $\pi_{\ell'}: \operatorname{Bl}_{\ell'}X \rightarrow \bP^2_{x_0,x_1,x_2}$.  
We describe the involution $\sigma$ on $E$ and on $\widetilde{C}$ respectively using Equation (\ref{eqn:X'}). 
To this end, let us consider a point $x$ on $\bP^2_{x_0,x_1,x_2} = (x_3=x_4=0)\subset \mathbb P^4$.  The plane it parameterizes is the span 
$\langle \ell', x \rangle$.  Since $\ell'$ and $x$ are preserved by $\tau$ ($x_4\mapsto -x_4$), then so is the span $\langle \ell', x \rangle$.  Assume now that $x$ is in $L$ or $C$, but not both.  Then the intersection of $\langle \ell', x \rangle$ with $X$ consists of three distinct lines, $\ell'\cup m\cup m'$, with $m$ and $m'$ corresponding to points of $E\cup \widetilde C$.  
Since $\langle \ell', x \rangle$ is preserved by $\tau$, 
it follows that either $\tau$ interchanges $m$ and $m'$, or it takes $m$ to $m$ and $m'$ to $m'$.  In other words, 
we see that $\sigma$ either acts as the identity on a point of $E\cup \widetilde C$, or by $\iota$.  
Now, since the only fixed points of the action of $\tau$ on $\ell'$ are the points $p=[0,0,0,0,1]$ and $p'=[0,0,0,1,0]$, the only way $m$ and $m'$ are not interchanged is if one of them passes through $p$ (which forces the other to pass through $p'$). It is easy to see from Equation \eqref{eqn:X'} $m$ or $m'$ passes through $p$ if and only if $l(x)=0$.  
Thus, the induced involution $\sigma$ on $\widetilde{C}$ (respectively, $E$) coincides with the covering involution $\iota$ (respectively, the identity map), as claimed. 

It follows that the isogeny $\phi$ is equivariant with respect to the involution $\jmath=(1,\iota)$ on  the product $E \times (J(\widetilde{C})/\pi^*J(C))$.  Since $\pi^*J(C)=\operatorname{Im}(1+\iota)$ on $J(\widetilde{C})$,
  it follows that $\operatorname{Im}(1+\jmath)=E\times\{0\}$.  
Similarly, $\operatorname{Im}(1-\jmath)=\{0\}\times J(\widetilde{C})/\pi^*J(C)$.  From a dimension count, we have then that $\phi$ induces isogenies $E\times \{0\}\to P(\widetilde D,D)^{\iota}$ and $\{0\}\times(J(\widetilde{C})/\pi^*J(C))\to P(\widetilde D,D)^{-\iota}$.  

Since $H \cap (E \times \{0\}) = H \cap (\{0\} \times (J(\widetilde{C})/\pi^*J(C))) = \{(0,0)\}$, these isogenies are isomorphisms. 
That these are isomorphisms of polarized abelian varieties follows from the fact that $\phi$ is a morphism of polarized abelian varieties, and 
the fact that we are taking the product polarization on $E\times (J(\widetilde{C})/\pi^*J(C))$.  In other words, if we take $\Xi$, pull it back to $E\times (J(\widetilde{C})/\pi^*J(C))$, and then restrict to a component, we get the same thing as if we take $\Xi$, restrict it to an invariant (respectively, anti-invariant) part, and then pull-back via this isomorphism.  
\end{proof}

We summarize the situation for Eckardt cubic threefolds in the following theorem.  

\begin{theorem}[Projecting from a line  $p\in \ell' \subset X$] \label{T:conedecomp}
Let $(X,p)$ be an Eckardt cubic threefold, and let $\ell'\subset X$ be a line passing through $p$ such that $\ell'$ does not meet any lines in the associated cubic surface $S$.  Let $E\subset S$ be the hyperplane section, and let $\pi:\widetilde C\to C$ be the restriction of the descriminant cover to the irreducible component $C$ that is a smooth plane quartic.  We identify $J(\widetilde{C})/\pi^*J(C)$ as the dual abelian variety to $P(\widetilde C,C)$, with the dual polarization.  There is an isogeny 
$$
\phi:E\times (J(\widetilde{C})/\pi^*J(C))\longrightarrow JX
$$
with $\ker \phi \cong (\mathbb Z/2\mathbb Z)^2$.  

Moreover, with respect to the action of $\tau$ on $(JX,\Theta_X)$, the isogeny $\phi$ induces  isomorphisms of polarized abelian varieties  $JX^\tau \cong E$ and $JX^{-\tau}\cong J(\widetilde{C})/\pi^*J(C)$.
\end{theorem}

\begin{proof}
Since $JX\cong P(\widetilde D,D)$, and the action of $\tau$ on $JX$ is identified with the action of $\sigma$ on $P(\widetilde D,D)$, 
this is an immediate consequence of \propositionref{prop:coneprym}. 
\end{proof}

\section{Recalling the fiber of the Prym map over the  cubic threefold locus} \label{S:DS}

In this section, we recall  Donagi--Smith's  result on the fiber of the Prym map over the intermediate Jacobian locus, as some details in various steps in that proof are central to our proof of \theoremref{thm:prymfiber} on the fiber of a related Prym map for Eckardt cubic threefolds.

Recall that an admissible double cover of curves is called an \emph{allowable} double cover if the associated Prym is compact (e.g., \cite[\S I.1.3]{DS_prym} and \cite[(**), p.173]{Beauville_schottky}).
The first result we will review is:

\begin{theorem}[{\cite[Thm.~V.1.1]{DS_prym}}] \label{thm:DSprymfiber}
Let $X$ be a smooth cubic threefold. The fiber of the Prym map for allowable double covers of genus $6$ curves
$$
\calP_{6}: \overline{\calR}_{6}^a \longrightarrow \calA_5
$$
over the intermediate Jacobian $JX\in \mathcal A_5$ is isomorphic to the Fano surface of lines $F(X)$:
$$
\mathcal P_6^{-1}(JX)\cong  F(X).
$$
\end{theorem}

\begin{rem}
We state \theoremref{thm:DSprymfiber} for the moduli stacks.  For the coarse moduli spaces, we have $\mathcal P^{-1}_6\left(JX  \right)\cong F(X)/\operatorname{Aut}(X)$.  
\end{rem}

  Let us note here that from the proof of \theoremref{thm:DSprymfiber}, one obtains:

\begin{cor}[{\cite[Thm.~I.2.1]{DS_prym}}]\label{C:DSdegree}
The Prym map
$$
\calP_{6}: \overline{\calR}_{6}^a \longrightarrow \calA_5
$$
has degree $27$, corresponding to the $27$ lines contained in a general hyperplane section of a smooth cubic threefold $X$. 
 \end{cor}

\begin{rem}
The proof of \theoremref{thm:DSprymfiber} use the fact that the degree of the Prym map $\mathcal P_6$ is equal to $27$, so that  \corollaryref{C:DSdegree} does not provide a new derivation of the degree.  The point is that one recovers the degree of the Prym map in terms of the geometry of cubic threefolds. 
\end{rem}

\subsection{Strategy for computing the fiber}\label{S:DS-method}
We start by reviewing Donagi--Smith's strategy for computing the degree of a generically finite morphism by studying the differential along positive dimensional fibers.  We provide here a slightly different presentation, which is perhaps more streamlined for our purposes. 

Suppose that $f:Y'\to Y$ is a proper, surjective, generically finite morphism of  integral separated  schemes of finite type over a field.  For any closed subscheme $Z\subset Y$, we have the fibered product  diagram.
\begin{equation}\label{E:FultonSegre}
\xymatrix{
Z' \ar@{^(->}[r] \ar[d]^{f'}& Y' \ar[d]^f\\
Z \ar@{^(->}[r]& Y
}
\end{equation}
  Let $\widetilde Y\to Y$ (respectively, $\widetilde Y'\to Y'$) be the blow-up of $Y$ along $Z$ (respectively, $Y'$ along $Z'$), with exceptional divisor $\widetilde Z=\mathbb PC_ZY$ (respectively, $\widetilde Z'=\mathbb PC_{Z'}Y'$). Setting $\widetilde f:\widetilde Y'\to \widetilde Y$ be the induced morphism on blow-ups, then 
 the induced morphism of exceptional divisors $\widetilde f|_{\widetilde Z'}:\widetilde Z'\to \widetilde Z$ is generically finite, and
\begin{equation}\label{E:degComp}
\deg f=\deg (\widetilde f|_{\widetilde Z'}:\widetilde Z'\to \widetilde Z);
\end{equation}
i.e., the degree of the map $f$ is equal to the degree of the blow-up when restricted to the exceptional divisors.  
Indeed, since  $\deg f=\deg (\widetilde f:\widetilde Y'\to \widetilde Y)$, we can immediately reduce to the case where $Z$ and $Z'$ in \eqref{E:FultonSegre} are Cartier divisors, in which case we are trying to show $\deg f= \deg f'$.  Then
using \cite[Prop.~4.2(a), Cor.~4.2.1]{fulton}, we have 
 have $f'_*((1-[Z'])\cap [Z'])=f'_*s(Z',Y')= (\deg f) s(Z,Y)=  (\deg f) ((1-[Z])\cap [Z])$. 
 Looking at the component of $[Z]$, and the fact that by definition $f_*[Z']=(\deg f')[Z]$, we have $\deg f=\deg f'$.  

Conversely, once one knows the degree of a generically finite map, one can use this to study the fibers:

\begin{lem}[{Identifying fibers \cite[\S I.3]{DS_prym}}]\label{L:DS-fiber}
In the notation above, assume $Y$ is geometrically unibranch (see  \cite[\href{https://stacks.math.columbia.edu/tag/0BQ2}{Tag 0BQ2}, \href{https://stacks.math.columbia.edu/tag/0BPZ}{Tag 0BPZ}]{stacks-project}; e.g., $Y$ is normal \cite[\href{https://stacks.math.columbia.edu/tag/0BQ3}{Tag 0BQ3}]{stacks-project}), $Z$ is integral, and $Z'_0\subset Z'$ is a \emph{connected component} of $Z'$ dominating $Z$, with $Z_0'$ integral, and with the generic fiber of $Z'_0\to Z$ connected. Let $\widetilde Z'_0$ be the exceptional divisor in the blow-up of $Y'$ along $Z_0'$, and let $\widetilde f_0:\widetilde Z_0'\dashrightarrow   \widetilde Z'\to \widetilde Z$ be the associated rational map. If $\deg \widetilde f_0\ge \deg f$, then  $\deg \widetilde f_0=\deg f$, and  $Z'_0=Z'$.  
\end{lem}

\begin{proof}
Clearly $\deg \widetilde f_0\le \deg (\widetilde f|_{\widetilde Z'}:\widetilde Z'\to \widetilde Z)=\deg \widetilde f=\deg f$ \eqref{E:degComp}, giving the stated equality of degrees.  On the other hand, once we have $\deg \widetilde f_0=\deg f$, we see that $Z'$ cannot have any other irreducible components dominating $Z$ (otherwise they would increase the degree count of $\widetilde f:\widetilde Z'\to \widetilde Z$).     The assumption that the generic fiber of $Z'_0\to Z$ is connected, together with the basic \lemmaref{L:CC-fibers} below,  rules out components of $Z'$ that do not dominate $Z$.
\end{proof}

\begin{rem}
The hypothesis that  $Z_0'$ be a connected component of $Z$ is necessary (i.e., it is not enough to assume only that $Z_0'$ be a reduced irreducible component of $Z'$).  Indeed, let $f:Y'\to Y$ be the blow-up of a smooth surface at a point, and let $Z\subset Y$ be a smooth irreducible curve through the point.  Taking $Z_0'$ to be the strict transform of $Z$, we have that $\deg (\widetilde f_0:Z_0'\to Z)=1\ge \deg f$; however, $Z_0'\ne Z'$.    The hypothesis that $Y$ be geometrically unibranch is also necessary.  Indeed, let $f:Y'\to Y$ be the normalization of an integral nodal curve, let $Z\subset Y$ be a node, and let $Z_0'$ be one of the points in $Y'$ lying over the node $Z$.  Then $\deg \widetilde f_0=1\ge \deg f$, but $Z_0'\ne Z'$. 
\end{rem}

\begin{rem}
In this paper we will typically be applying \lemmaref{L:DS-fiber} to the quasi-projective coarse moduli spaces of certain smooth Deligne--Mumford (DM) stacks.
The key point is that the coarse moduli space of a smooth DM stack, locally being the quotient of a smooth space by a finite group, is normal. 
\end{rem}
 
We now recall a variation on a standard result  \cite[Thm.~4.17(ii)]{DM69} (\cite[Lem.~0BUI]{stacks-project})  on connected components of fibers of maps:

\begin{lem}\label{L:CC-fibers}
Let $f:X\to S$ be a proper surjective morphism of finite presentation between integral schemes, with $S$ geometrically unibranch. Let $n_{X/S}$ be the function on $S$ counting the numbers of geometric connected components of fibers of $f$ (e.g., \cite[Lem.~055F]{stacks-project}). Then $n_{X/S}$ is lower semi-continuous.
\end{lem}

\begin{proof}
The proof is essentially identical to that of the standard result  \cite[Lem.~0BUI]{stacks-project}. For brevity, the proof is left to the reader.
\end{proof}

\subsection{The Donagi--Smith argument for cubic threefolds}

Suppose now we have a cubic threefold $X\subset \mathbb P^4$ together with a  line $\ell\subset X$ giving rise to the odd connected \'etale double cover $\pi:\widetilde D\to D$ of the smooth  plane quintic discriminant $D\subset \mathbb P^2$.  As in the discussion above, assume that $X=(F=0)$ and $D=(Q=0)$.

Let $\overline{\mathcal R}^a_{\mathcal Q}$ denote the moduli space of allowable odd double covers of plane quintics, and let $\mathcal C$ denote the moduli space of cubic threefolds.
From the commutative diagram (which at this point we have not proven is cartesian)
$$
\xymatrix{
\overline{\mathcal R}^a_{\mathcal Q}\ar[r] \ar[d]& \overline{\mathcal R}^a_6 \ar[d]\\
\mathcal C\ar[r]& \mathcal A_5
}
$$
we obtain a commutative diagram of codifferentials:
\begin{equation}\label{E:CubicCoDiffDi}
\xymatrix@C=1.5em{
T^\vee_\pi \overline{\mathcal R}^a_{\mathcal Q}& T_\pi^\vee \overline{\mathcal R}^a_6 \ar@{->>}[l]& (T_\pi \overline{\mathcal R}^a_6/T_\pi \overline{\mathcal R}^a_{\mathcal Q})^\vee \ar@{_(->}[l] & R^4_Q & H^0(D,\mathcal O_D(4)) \ar@{->>}[l]& J^4_Q=\mathbb C^3 \ar@{_(->}[l]\\
T^\vee_X\mathcal C \ar[u]& T_{JX}^\vee \mathcal A_5 \ar@{->>}[l] \ar[u]& (T_{JX}\mathcal A_5/T_X\mathcal C)^\vee \ar@{_(->}[l] \ar[u]& R_F^2  \ar[u]& H^0(\mathbb P^4,\mathcal O_{\mathbb P^4}(2)) \ar@{->>}[l]  \ar[u]& J^2_F=\mathbb C^5  \ar@{_(->}[l] \ar[u]
}
\end{equation}
Here $R_F=\mathbb C[x_0,\dots,x_4]/(\frac{\partial F}{\partial x_0},\dots, \frac{\partial F}{\partial x_4})$ (respectively, $J_F=(\frac{\partial F}{\partial x_0},\dots, \frac{\partial F}{\partial x_4})$) is defined to be the Jacobian ring (respectively, Jacobian ideal), with the superscript denoting the vector space of forms of a given degree; we define $R_Q$ and $J_Q$ similarly.

In fact, we will need to consider the case where $\pi:\widetilde D\to D$ is a pseudo-double cover of a nodal plane quintic.  We would like to have a diagram as above, but this requires a bit more work.  
Recall that for a nodal curve $D$ there is a natural map 
$ j:\Omega_D\to \omega_D$
  with torsion kernel and cokernel (e.g., \cite[\S IV.2.3.3]{DS_prym}).  This induces a natural map $j:\Omega_D\otimes \omega_D\to \omega_D^{\otimes 2}$.

\begin{lem}\label{L:DSV15-explained}
Suppose $X=(F=0)\subset \mathbb P^4$ is a smooth cubic threefold, $\ell\subset X$ is a line, and   $\pi:\widetilde D\to D$ is the associated odd pseudo-double cover of the discriminant  plane quintic $D=(Q=0)\subset \mathbb P^2$. 
There is a commutative diagram of codifferentials:
\begin{equation}\label{E:CubicCoDiffDiAll}
\xymatrix@C=1.5em{
&&&R^4_Q&H^0(D,\mathcal O_D(4)) \ar@{->>}[l]& J^4_Q=\mathbb C^3 \ar@{_(->}[l]\\
T^\vee_\pi \overline{\mathcal R}^a_{\mathcal Q}& T_\pi^\vee \overline{\mathcal R}^a_6 \ar@{->>}[l]& (T_\pi \overline{\mathcal R}^a_6/T_\pi \overline{\mathcal R}^a_{\mathcal Q})^\vee \ar@{_(->}[l] & (R_Q^5)^\vee \ar[u]& H^0(D,\Omega_D\otimes \mathcal O_D(2)) \ar@{->>}[l] \ar[u]^{j_*}& K=\mathbb C^3 \ar@{_(->}[l] \ar[u]\\
T^\vee_X\mathcal C \ar[u]& T_{JX}^\vee \mathcal A_5 \ar@{->>}[l] \ar[u]& (T_{JX}\mathcal A_5/T_X\mathcal C)^\vee \ar@{_(->}[l] \ar[u]& R_F^2  \ar[u]& H^0(\mathbb P^4,\mathcal O_{\mathbb P^4}(2)) \ar@{->>}[l]  \ar[u]& J^2_F=\mathbb C^5  \ar@{_(->}[l] \ar[u]
}
\end{equation}
where $K$ is defined as the kernel of the map $ H^0(D,\Omega_D\otimes \mathcal O_D(2))\to (R^5_Q)^\vee$, and the morphism $K\to J^4_Q$ is an isomorphism.
\end{lem}

\begin{proof}
We direct the reader to \cite[\S II.5.1]{DS_prym}.
\end{proof}

For convenience, we now recall \cite[Lem.~V.1.5, p.88]{DS_prym} describing the map $J^2_F\to J^4_Q$ in  \eqref{E:CubicCoDiffDi} and \eqref{E:CubicCoDiffDiAll}.

\begin{lem}[{\cite[Lem.~V.1.5, p.88]{DS_prym}}] \label{L:DSV15}
Let $X\subset \mathbb P\mathbb C^5 \cong \mathbb P^4$ be a smooth cubic threefold given by $F=0$, and let $\ell\subset X$ be a line giving rise to 
the odd pseudo-double cover $\pi:\widetilde D\to D$ of the discriminant  plane quintic $D\subset \mathbb P\mathbb C^3  \cong \mathbb P^2$ given by $Q=0$.   Take coordinates on $\mathbb C^5$ such that $\ell=(x_0=x_1=x_2=0)$, so that projection from $\ell$ gives the  map  $\pi_\ell:\mathbb P\mathbb C^5\dashrightarrow  \mathbb P\mathbb C^3$, $[x_0,x_1,x_2,x_3,x_4]\mapsto [x_0,x_1,x_2]$.  
 Under the identifications
$$
\xymatrix@R=.1em{
\mathbb C^3 \ar[r]^\sim& J^4_Q& \mathbb C^5\ar[r]^\sim& J^2_F\\
(a_0,a_1,a_2)\ar@{|->}[r]& \sum_{i=0}^2 a_i\frac{\partial Q}{\partial x_i}& (a_0,a_1,a_2,a_3,a_4)\ar@{|->}[r] &\sum_{i=0}^4 a_i\frac{\partial F}{\partial x_i}
} 
$$
the projectivization of the morphism $J^2_F\to J^4_Q$ of \eqref{E:CubicCoDiffDiAll} is the rational map $\pi_\ell$.   The projectivization of the dual map $(J^4_Q)^\vee\to (J^2_F)^\vee$ is therefore the inclusion $(\mathbb P^2_\ell)^\vee \subset (\mathbb P^4)^\vee$, where we are denoting by $(\mathbb P^2_\ell)^\vee$ the $2$-plane in $(\mathbb P^4)^\vee$ corresponding to hyperplanes containing $\ell$.   \qed
\end{lem}

We now recall the completion of the proof of \theoremref{thm:DSprymfiber}:

\begin{proof}[Donagi--Smith's proof of \theoremref{thm:DSprymfiber}]
We apply the strategy of \eqref{E:FultonSegre} and \lemmaref{L:DS-fiber} to the diagram
\begin{equation}
\xymatrix{
\overline {\mathcal R}_{\mathcal Q}^a\ar[r] \ar[rd]& \mathcal C\times_{\mathcal A_5}\overline{\mathcal R}_6^a\ar[r] \ar[d]& \overline {\mathcal R}_6^a \ar[d]^{\mathcal P_6}\\
& \mathcal C\ar[r]& \mathcal A_5
}
\end{equation}
where $\overline R_{\mathcal Q}^a$, the moduli space of odd allowable double covers of nodal plane quintics,  plays the role of $Z_0'$.
\lemmaref{L:DSV15} implies that the differential of the Prym map  $\overline{\mathcal R}^a_6\to \mathcal A_5$ at a cover $(\pi:\widetilde D\to D)\in \overline R_{\mathcal Q}^a$ has kernel exactly of dimension $2$.  Thus at $\pi$, the fiber of $\mathcal P_6$ over $P(\widetilde D,D)=JX$ is exactly the Fano surface of lines $F(X)$.  Since this holds at every cover $\pi$, this implies that  $\overline {\mathcal R}_{\mathcal Q}^a$ is a connected  component of the fiber over $\mathcal C$, with generic fiber of $\overline {\mathcal R}_{\mathcal Q}^a$ over $\mathcal C$ connected.

  At the same time, let us use  \lemmaref{L:DSV15} to compute the degree of the map of exceptional divisors $\widetilde{\overline{\mathcal R}_{\mathcal Q}^a}\to \widetilde {\mathcal C}$; here $\widetilde{\overline{\mathcal R}_{\mathcal Q}^a}$ (respectively~$\widetilde {\mathcal C}$) is the exceptional divisor in the blow-up of $\overline{\mathcal R}_6^a$ (respectively~$\mathcal A_5$) along the locus $\overline{\mathcal R}^a_{\mathcal Q}$ (respectively~$\mathcal C$).  For this, we can pick a general cubic threefold $X$, and a general point in the projectivization of the fiber of the normal bundle, corresponding, by \lemmaref{L:DSV15}, to a general hyperplane $H\subset \mathbb P^4$.  \lemmaref{L:DSV15} says that the points lying above this chosen point correspond to those lines  $\ell\subset X$ that lie in the hyperplane section $H$.  Since $X\cap H$ is a smooth cubic surface, there are $27$ such lines.

Since the degree of the Prym map is $27$ \cite[Thm.~I.2.1]{DS_prym}, we can conclude from \lemmaref{L:DS-fiber} that 
$\overline {\mathcal R}_{\mathcal Q}^a=\mathcal C\times_{\mathcal A_5}\overline{\mathcal R}_6^a$, and we are done.
\end{proof}

\begin{proof}[Donagi--Smith's proof of \corollaryref{C:DSdegree}]
This now follows immediately from \eqref{E:degComp}, and the proof of \theoremref{thm:DSprymfiber}.  
\end{proof}

\section{Fiber of the Prym map $\calR_{3,4} \rightarrow \calA_4^{(1,2,2,2)}$ over the Eckardt cubic threefold locus}\label{S:PrymFiber}

Our goal is to describe the fiber of the generically finite degree $3$  Prym map $\calP_{3,4}: \calR_{3,4} \rightarrow \calA_4^{(1,2,2,2)}$ over the Eckardt cubic locus; recall that $\mathcal R_{3,4}$ is the moduli space of double covers of smooth genus $3$ curves branched at $4$ points, and $\calA_4^{(1,2,2,2)}$ is the moduli space of abelian varieties of dimension $4$ with a polarization of type $(1,2,2,2)$.  

In order to study the fibers of $\mathcal P_{3,4}$, it is convenient to consider a partial compactification $\overline{\mathcal R}^a_{3,4}$ of $\mathcal R_{3,4}$, over which the Prym map extends to a proper map $\mathcal P_{3,4}:\overline{\mathcal R}^a_{3,4}\to \mathcal A_4^{(1,2,2,2)}$. As mentioned in \S \ref{S:CubSolid}, in the \'etale case this is due to Beauville \cite{Beauville_schottky}; in the case of branched covers, one should technically use the Harris--Mumford \cite{HM82} notion of admissible covers (which we will call \emph{branched admissible covers}), and in fact, we prefer to use the Abramovich--Corti--Vistoli description \cite{ACV_twisted}. More precisely, we denote by $\overline{\mathcal R}_{g,r}$ the moduli space of branched admissible double covers where the base curve is a stable curve of genus $g$ with $r$ unordered marked points, and the branch points of the cover are given by those $r$ marked points; in the language of \cite[p.3560 and Prop.~4.2.2]{ACV_twisted}, the space $\overline {\mathcal R}_{g,r}$  is the quotient of $\mathscr B^{\operatorname{bal}}_{g,r}(\mathscr S_2)$ by the symmetric group $\mathscr S_r$ acting by permuting the labeling of the $r$ points. For any branched admissible double cover, the connected component of the kernel of the norm map defines a semi-abelian  Prym variety \cite[\S 5]{Beauville_schottky} \cite[\S 1]{ABH}. The condition that the semi-abelian variety be compact can be described as a combinatorial condition on the dual graph \cite[Prop.~1.3]{ABH}, which agrees with the combinatorial condition in the \'etale case given by Beauville  \cite[p.173, (**)]{Beauville_schottky}; i.e., the number of components of the cover that are interchanged by the covering involution must be equal to the number of nodes of the cover that are interchanged. We call such  branched admissible double covers \emph{branched allowable double covers}. From the deformation theory of $\overline{\mathcal R}_{g,r}$  \cite[\S 3.0.4]{ACV_twisted} it is clear that  a small deformation of a branched allowable double cover will be allowable, so that there is an open substack $\overline{\mathcal R}^a_{g,r}\subset \overline{\mathcal R}_{g,r}$ of branched allowable double covers. The Prym construction then gives a period map $$\mathcal P_{g,r}:\overline{\mathcal R}^a_{g,r}\longrightarrow \mathcal A_{g-1+\frac{r}{2}}^D$$ where $D=(1,\dots,1)$ if $r=0,2$, and otherwise $D=(1,\dots,1,2\dots,2)$ with $\frac{r}{2}-1$ copies of $1$; the argument  of \cite[Prop.~6.3]{Beauville_schottky} shows that this extended period map is proper. 

Now to describe the fibers of $\mathcal P_{3,4}$, 
recall that given an Eckart cubic threefold $(X,p)$, there is an associated cubic surface $S$, together with a smooth  hyperplane section $E\subset S$. This elliptic curve is naturally contained in the Fano surface $F(X)$ of lines, with $E$ parameterizing the  $\tau$-invariant lines in $X$ through the Eckardt point $p$; i.e., we may view  $E$ as the $1$-dimensional component of the $\tau$-fixed locus  of the Fano surface $F(X)$, and $E$ is isomorphic to the elliptic curve $S \cap \Pi$; see \lemmaref{lem:invline}.

\begin{theorem} \label{thm:prymfiber}
Let $(X,p)$ be an Eckardt cubic threefold with involution $\tau$.  The fiber of the Prym map for allowable
double covers of genus $3$ curves branched at $4$ points
$$
\calP_{3,4}: \overline{\calR}_{3,4}^a \longrightarrow \calA_4^{(1,2,2,2)}
$$
over the dual abelian variety  $(JX^{-\tau})^\vee \in \calA_4^{(1,2,2,2)}$ of the anti-invariant part of the intermediate Jacobian $JX$ is isomorphic to the elliptic curve $E$:
$$
\mathcal P_{3,4}^{-1}\left((JX^{-\tau})^\vee  \right)\cong E \subset F(X).
$$
\end{theorem}

\begin{rem}
We state this theorem for the moduli stacks.  For the coarse moduli spaces, we have $\mathcal P_{3,4}^{-1}\left((JX^{-\tau})^\vee  \right)\cong E/\operatorname{Aut}(X,p)$; note that the Eckardt automorphism $\tau\in \operatorname{Aut}(X,p)$ acts trivially on $E$.  
\end{rem}

We prove this theorem over the course of this section.  Let us note here that we will also obtain the following corollary: 

\begin{cor}[{\cite{BCV_pav122,NR_pav122}, \cite[Thm.~0.3]{NO_prymtorelli}}]\label{C:degree}
The Prym map
$$
\calP_{3,4}: \overline{\calR}_{3,4}^a \longrightarrow \calA_4^{(1,2,2,2)}
$$
has degree $3$, corresponding to the three lines through the Eckardt point $p$ contained in a general hyperplane section of an Eckardt cubic threefold $(X,p)$ through the Eckardt point. 
\end{cor}

\begin{rem}
The proof of \theoremref{thm:prymfiber} uses the fact that the degree of the Prym map $\mathcal P_{3,4}$ is equal to $3$, which is due to \cite{BCV_pav122,NR_pav122, NO_prymtorelli}, so that  \corollaryref{C:degree} does not provide a new derivation of the degree.  The new result is that one recovers the degree of the Prym map in terms of the geometry of Eckardt cubic threefolds. 
\end{rem}

To prove \theoremref{thm:prymfiber}, we use the same strategy as in the previous section.  Namely, 
we apply the strategy of \eqref{E:FultonSegre} and \lemmaref{L:DS-fiber} to the diagram
\begin{equation}\label{E:EckFiberMainDi}
\xymatrix{
\overline {\mathcal R}_{\mathcal Q, \mathcal E}^a\ar[r] \ar[rd]& \mathcal C_{\mathcal E}\times_{\mathcal A_4^{(1,2,2,2)}}\overline{\mathcal R}_{3,4}^a\ar[r] \ar[d]& \overline {\mathcal R}_{3,4}^a \ar[d]^{\mathcal P_{3,4}}\\
& \mathcal C_{\mathcal E}\ar[r]& \mathcal A_{4}^{(1,2,2,2)}
}
\end{equation}
where the image of $\mathcal C_{\mathcal E}$ plays the role of $Z$, and the image of $\overline {\mathcal R}_{\mathcal Q, \mathcal E}^a$ plays the role of $Z_0'$.
Here $\mathcal C_{\mathcal E}$ is the moduli space of Eckardt cubic threefolds $(X, p)$, and  ${\mathcal R}_{\mathcal Q,\mathcal E}$ is the moduli space of double covers $\pi:\widetilde C\to C$ over smooth plane quartics branched at the four points of intersection of a transverse line, determined by an odd theta characteristic (see \lemmaref{L:kappaCEck}).  In other words,  $\overline {\mathcal R}_{\mathcal Q, \mathcal E}^a$ is the space of covers arising from projecting an Eckardt cubic threefold from a line through the Eckardt point, and restricting the cover to the plane quartic in the discriminant.  

We therefore start by computing the differentials of the morphisms above.

\subsection{Differential to the period map for Eckardt cubic threefolds}\label{S:diffEck}

Let $\mathcal C_{\mathcal E}$ be the moduli space of Eckardt cubic threefolds $(X, p)$. Assuming that $(X,p)$ is a smooth Eckardt cubic threefold defined by $F=0$ as in \eqref{eqn:eckardt}, and $R_F$ is defined to be the Jacobian ring,  we saw in the proof of \lemmaref{lem:dimpm} that $H^{2,1}(X)^{-\tau}=(R_F^1)^{-\tau}$  and $H^{1,2}(X)^{-\tau}=(R_F^4)^{-\tau}$, 
so that we also have  
$$T_{(JX^{-\tau})^\vee} \calA_4^{(1,2,2,2)}=\operatorname{SymHom}(H^{2,1}(X)^{-\tau},H^{1,2}(X)^{-\tau})=
\operatorname{SymHom}((R_F^1)^{-\tau},(R_F^4)^{-\tau}).
$$
At the same time,  we have  
$$
T _{(X,p)}\mathcal C_{\mathcal E}=(T_X\mathcal C)^{\tau}=(R_F^3)^{\tau}.
$$
Therefore, with regards to the period map
$$
\mathcal C_{\mathcal E}\longrightarrow \mathcal A_4^{(1,2,2,2)}
$$
we obtain a map of tangent spaces 
$$
T_{(X,p)}\mathcal C_{\mathcal E}\longrightarrow T_{(JX^{-\tau})^\vee} \calA_4^{(1,2,2,2)}
$$
\begin{equation}\label{E:EckCubDiffPer}
(R^3_F)^{\tau}\longrightarrow \operatorname{SymHom}((R^1_F)^{-\tau},(R^4_F)^{-\tau}).
\end{equation}

Using Macaulay's Theorem to identify $R^4_F \cong (R^1_F)^\vee$, 
our differential is 
$$
(R^3_F)^{\tau}\longrightarrow \operatorname{Sym}^2((R^1_F)^{-\tau})^\vee,
$$
and the dual map, the codifferential, is canonically
$$
H^0(\mathbb P^3,\mathcal O_{\mathbb P^3}(2))=\operatorname{Sym}^2((R^1_F)^{\tau})\longrightarrow (R^2_F)^{\tau}.
$$
Here we are using the identification $R^1_F=\mathbb C\langle x_0, \dots,x_4\rangle$, where the action of $\tau$ is given by $x_4\mapsto -x_4$, so that $(R^1_F)^{\tau}=\mathbb C\langle x_0,\dots,x_3\rangle$.  
We  see that the codifferential is the map that takes a quadric in $\mathbb P^3=(x_4=0)$, views it as an invariant quadric on $\mathbb P^4$, and then sends it to its class in the invariant Jacobian ring. 
Algebraically, the kernel of this map is clearly $J^2_F\cap \operatorname{Sym}^2\mathbb C\langle x_0,\dots,x_3\rangle$.  
Using the equation \eqref{eqn:eckardt} for $F$, we see that 
$J^2_F$ consists of quadrics of the form $a_0(\frac{\partial f}{\partial x_0}+\frac{\partial l}{\partial x_0}x_4^2)+\cdots +a_3(\frac{\partial f}{\partial x_3}+\frac{\partial l}{\partial x_3}x_4^2)+2a_4lx_4$.  In order for this not to contain $x_4$, we must have $a_4=0$, and also $0=\sum_{i=0}^3 a_i\frac{\partial l}{\partial x_i}=l(a_0,\dots,a_3)$.  
We therefore have:
\begin{equation}\label{E:CokerEckCub}
(T_{(JX^{-\tau})^\vee}\mathcal A_4^{(1,2,2,2)}/T_{(X,p)}\mathcal C_{\mathcal E})^\vee \cong J^2_F\cap \operatorname{Sym}^2\mathbb C\langle x_0,\dots,x_3\rangle \cong (l=a_4=0) \cong \mathbb C^3.
\end{equation}
 We identify $J^2_F\cap \operatorname{Sym}^2\mathbb C\langle x_0,\dots,x_3\rangle$ with $(l=a_4=0) \cong \mathbb C^3
$ via $(a_0,\dots,a_3)\mapsto \sum_{i=0}^3 a_i\frac{\partial F}{\partial x_i}$, and call these 
the quadratics $X_t$ polar to points $t=(a_0,\dots,a_3,0)$ of $(l=a_4=0)\subset \bP^4$ with respect to $X$.

\begin{cor}[Infinitesimal Torelli theorem]\label{C:InfTorHdg}
The differential of the period map $\mathcal C_{\mathcal E}\to \mathcal A^{(1,2,2,2)}_4$ is injective.
\end{cor}

\begin{proof}
We consider \eqref{E:EckCubDiffPer}.  
Since $\dim (R^3_F)^{\tau}=7$, $\dim \operatorname{Sym}^2 (R^1_F)^{-\tau}=10$, and the cokernel of  \eqref{E:EckCubDiffPer} has dimension $3$ \eqref{E:CokerEckCub}, it follows that \eqref{E:EckCubDiffPer} is injective.
\end{proof}

\subsection{Differential of the inclusion of  the moduli space of double covers of smooth plane quartics branched at four collinear points into the moduli space of double covers of genus $3$ curves branched at four points}  

We let ${\mathcal R}_{\mathcal Q,\mathcal E}$ be the moduli space of double covers $\pi:\widetilde C\to C$ over smooth plane quartics branched at the four points of intersection of a transverse line, determined by an odd theta characteristic (see \lemmaref{L:kappaCEck}),  and we consider the inclusion
$$
{\mathcal R}_{\mathcal Q,\mathcal E}\longrightarrow \mathcal R_{3,4}.
$$
Giving a cover $\pi:\widetilde C\to C$ in $\mathcal R_{\mathcal Q,\mathcal E}$ is equivalent to giving the plane quartic $C$, the branch points of the cover, which are themselves determined by a line $L$, and an odd square root of the branch divisor.  
Thus the space $\mathcal R_{\mathcal Q,\mathcal E}$ has dimension $6+2+0=8$.  
Moreover, the first order deformations of $\pi$ as covers in  $\mathcal R_{\mathcal Q,\mathcal E}$ are identified with the  first order deformations of the plane quintic $C\cup L$ that remain the union of a plane quartic and a line.
Up to change of coordinates, we may assume that the qunitic polynomial $Q$ defining $C\cup L$ is $Q=gx_0$ for some quartic $g$.  Then $T_\pi\mathcal R_{\mathcal Q,\mathcal E} \subset T_{\widetilde C\cup \widetilde L\to C\cup L} \mathcal R_{\mathcal Q}  \cong R^5_Q$ can be identified with the span of the quintic monomials divisible by $x_0$; indeed, these remain the union of a plane quartic and a line and then a dimension count gives the equality. Denote this space by $R_Q^5|x_0$.  More canonically, one can see that these deformations are exactly the locally trivial deformations of the plane quintic $C\cup L$, and are then identified with $H^0_{\mathfrak m}(R_Q)^5$, as in Remark \ref{R:loc_triv_quintic}, with dual $H^0_{\mathfrak m}(R_Q)^4$.

\begin{rem}[Locally trivial deformations of nodal plane quintics] \label{R:loc_triv_quintic}
If $D$ is a singular plane quintic, then the Jacobian ring no longer satisfies the duality of Macaulay's theorem.  However, 
there is still a natural map  
\begin{equation}\label{E:R5vee-R4}
(R^5_Q)^\vee\to R^4_Q
\end{equation}
 and the image is  identified with the dual of the space of locally trivial deformations.
More precisely, fix the maximal ideal  $\mathfrak m=(x_0,x_1,x_2)$  in the Jacobian ring $R_Q$.  There is an inclusion of graded rings \cite[(11)]{Sernesi_jacobian}
$ H^0_{\mathfrak m}(R_Q)\subset  R_Q$ 
and the ring $H^0_{\mathfrak m}(R_Q)$ satisfies a similar duality as in Macaulay's theorem; namely, there is an isomorphism $H^0_{\mathfrak m}(R_Q)^{i} \cong (H^0_{\mathfrak m}(R_Q)^{9-i})^\vee$ \cite[Thm.~3.4]{Sernesi_jacobian} (we note that this isomorphism is more subtle than in Macaulay's theorem; see \cite[Rmk.~3.5]{Sernesi_jacobian}). The space $H^0_{\mathfrak m}(R_Q)^5\subset R_Q^5$ is identified with the first order locally trivial deformations \cite[Cor.~2.2]{Sernesi_jacobian}.  In summary, we have morphisms $(R^5_Q)^\vee \to (H^0_{\mathfrak m}(R_Q)^5)^\vee \cong H^0_{\mathfrak m}(R_Q)^4\subset R^4_Q$, and the image of the composition $(R^5_Q)^\vee\to R^4_Q$ is identified with the dual of the space of locally trivial deformations, as claimed. 
\end{rem}

Another way to verify that $H^0_{\mathfrak m}(R_Q)^5 = R_Q^5|x_0$ is to use the  identification $H^0_{\mathfrak m}(R_Q) = J_f^{\mathrm{sat}}/J_f$ \cite[(5)]{Sernesi_jacobian}. In fact, by \cite[(11)]{Sernesi_jacobian} when $d$ is sufficiently large, the graded piece of our Jacobian ring $J_Q^d$ is of dimension $4$ (corresponding to the four nodes $C \cap L$ of the quintic $C\cup L$) and consists of monomials which are not divisible by $x_0$ (corresponding to the infinitesimal deformations smoothing the four nodes). Using a similar argument, we get that $H^0_{\mathfrak m}(R_Q)^4 = R_Q^4|x_0$.

At the same time, $T_{\pi}\mathcal R_{3,4}=H^1(C,TC\otimes \mathcal O_C(L))=H^0(C,\mathcal O_C(3))^\vee$, 
  and so the differential can be described as a map 
$$
T_\pi \mathcal R_{\mathcal Q,\mathcal E}\longrightarrow T_\pi \mathcal R_{3,4}
$$
$$
H^0_{\mathfrak m}(R_Q)^5\to H^0(C,\mathcal O_C(3))^\vee.
$$
By the above argument, the codifferential sits in a commutative diagram
\begin{equation}\label{E:x0Mult-di}
\xymatrix{
H^0(C,\mathcal O_C(3))\ar[rd]^{\cdot x_0} \ar[r]& H^0_{\mathfrak m}(R_Q)^4 \ar[d] \\
& R_Q^4
}
\end{equation}
where the diagonal map is 
given by multiplication by $x_0$, and the vertical map is defined in Remark \ref{R:loc_triv_quintic}. 

The kernel of the diagonal map is $x_0\mathbb C[x_0,x_1,x_2]_3\cap J^4_Q$.  
Since $J^4_Q=\mathbb C\langle g+x_0\frac{\partial g}{\partial x_0}, x_0\frac{\partial g}{\partial x_1}, x_0\frac{\partial g}{\partial x_2}\rangle$, we see that the kernel is the space consisting of linear combinations $a_0(g+x_0\frac{\partial g}{\partial x_0})+a_1( x_0\frac{\partial g}{\partial x_1})+a_2(x_0\frac{\partial g}{\partial x_2})$, with $a_0=0$. 
In other words,
\begin{equation}\label{E:CokerQuarLine}
R^4_Q/H^0(C,\mathcal O_C(3))\cong J^4_Q \cap x_0\mathbb C[x_0,x_1,x_2]_3  \cong \mathbb C^2.
\end{equation}
where the identification with $\mathbb C^2$ is given by $(0,a_1,a_2)\mapsto x_0\sum_{i=1}^2 a_i\frac{\partial g}{\partial x_i}$.

\begin{rem}
The condition $a_0=0$ is more canonically the condition $l(a_0,a_1,a_2)=0$.
\end{rem}

\subsection{Differential of the inclusion of the moduli space of allowable double covers of nodal plane quartics branched at four collinear points into the moduli space of allowable double covers of genus $3$ curves branched at $4$ points}  We have 
$$
T_\pi \overline{\mathcal R}^a_{\mathcal Q,\mathcal E}\longrightarrow T_\pi \overline{\mathcal R}^a_{3,4}
$$
$$
H^0_{\mathfrak m}(R_Q)^5\to H^0(C,\Omega_C\otimes \mathcal O_C(2))^\vee
$$
with codifferential sitting in a commutative diagram.
$$
\xymatrix{
H^0(C,\Omega_C\otimes \mathcal O_C(2))\ar[d]_{j_*} \ar[r]& H^0_{\mathfrak m}(R_Q)^4 \ar[d]\\
H^0(C,\mathcal O_C(3)) \ar[r]^{\cdot x_0}&  R_Q^4
}
$$
Commutativity follows from \eqref{E:x0Mult-di} and a degeneration argument (see e.g., 
\cite[\S IV.3.2]{DS_prym}).

\subsection{Differential of the Prym map $\mathcal P_{3,4}$ for smooth double covers}
The differential of the Prym map $\mathcal P_{3,4}:\mathcal R_{3,4}\to \mathcal A_4^{(1,2,2,2)}$ at a branched cover $\pi:\widetilde C\to C$ of a smooth plane quartic 
$$T_\pi \mathcal R_{3,4}\longrightarrow T_\pi \mathcal A_{4}^{(1,2,2,2)}$$
is canonically identified with 
$$
H^1(C,TC\otimes \mathcal O_C(1))\to \operatorname{Sym}^2H^0(C,\omega_C\otimes \eta_C)^\vee
$$
where $\eta_C$  is the square root of the branch divisor $Br$ determining the cover (i.e., $\eta_C^{\otimes 2}\cong \mathcal O_C(Br)$).  
The  codifferential is given as the cup product
$$
\operatorname{Sym}^2H^0(D,\omega_C\otimes \eta_C)\to H^0(C,\omega_C^{\otimes 2}\otimes \mathcal O_C(Br)).
$$
For the restriction of an odd pseudo-double cover $\widetilde C\cup \widetilde L\to C\cup L$ to $\pi:\widetilde C\to C$, this is the canonical restriction 
$$
H^0(\mathbb P^4,\mathcal O_{\mathbb P^4}(2))\to H^0(C,\mathcal O_C(3))
$$
of a quadric to the Prym canonical model $C\subset \mathbb P^4$.  Note that the Prym canonical model of $C$ is also the image of $C$ under the Prym canonical map of the plane quintic $C\cup L$ and its associated pseudo-double cover.

\subsection{Differential of the Prym map $\mathcal P_{3,4}$ for allowable double covers}
The differential of the Prym map $\mathcal P_{3,4}:\overline{\mathcal R}^a_{3,4}\to \mathcal A_4^{(1,2,2,2)}$ at an allowable double cover  $\pi:\widetilde C\to C$ of a nodal plane quartic 
$$T_\pi \overline{\mathcal R}^a_{3,4}\longrightarrow T_{\mathcal P_{3,4}(\pi)} \mathcal A_4^{(1,2,2,2)}$$
is canonically identified with 
$$
H^0(C,\Omega_C\otimes \omega_C\otimes \mathcal O_C(1))^\vee\to \operatorname{Sym}^2H^0(C,\omega_C\otimes \eta_C)^\vee
$$
where $\eta_C^{\otimes 2}\cong \mathcal O_C(Br)$ is the square root of the branch divisor determining the cover.  (For us, this will be  the locus where the line meets the plane quartic transversally.) 
The  codifferential is given as the cup product
$$
\operatorname{Sym}^2H^0(D,\omega_C\otimes \eta_C)\to H^0(C,\Omega_C\otimes \omega_C\otimes \mathcal O_C(Br)).
$$
For the restriction of an odd pseudo-double cover $\widetilde C\cup \widetilde L\to C\cup L$ to $\pi:\widetilde C\to C$, this is the following map. 
$$
\xymatrix{
H^0(\mathbb P^4,\mathcal O_{\mathbb P^4}(2)) \ar[r]  \ar[rd] &  H^0(C,\Omega_C\otimes \mathcal O_C(2)) \ar[d]_{j_*}\\
& H^0(C,\mathcal O_C(3))
}
$$

\subsection{Interlude connecting to the case of cubic threefolds}
Suppose now we have a smooth Eckardt cubic threefold $(X,p)$ together with a  line $\ell' \subset X$ passing through the Eckardt point $p$, giving rise to a pseudo double cover $\pi:\widetilde C\cup \widetilde L \to C\cup L$ of a plane quartic and a transverse line  $C\cup L\subset \mathbb P^2$.  As in the discussion above, assume that $X=(F=0)$ and $C\cup L=(Q=0)$.

For the commutative diagram (which at this point we have not proven is cartesian)
$$
\xymatrix{
\overline{\mathcal R}^a_{\mathcal Q,\mathcal E}\ar[r] \ar[d]& \overline{\mathcal R}^a_{3,4} \ar[d]\\
\mathcal C_{\mathcal E}\ar[r]& \mathcal A_4^{(1,2,2,2)}
}
$$
we have:

\begin{proposition}\label{P:stablequartic}
In the notation above, we have the commutative diagram of codifferentials:
\begin{equation}\label{E:EckCubicCoDiffDiAll}
\xymatrix@C=.5em{
&& &R^4_Q & H^0(\mathcal O_C(3)) \ar[l]_<>(0.5){\cdot x_0}& J^4_Q\cap x_0\mathbb C[x_0,x_1,x_2]_3=\mathbb C^2 \ar@{_(->}[l]\\
T^\vee_\pi \overline{\mathcal R}^a_{\mathcal Q,\mathcal E}& T_\pi^\vee \overline{\mathcal R}^a_{3,4} \ar@{->}[l]& (T_\pi \overline{\mathcal R}^a_{3,4}/T_\pi \overline{\mathcal R}^a_{\mathcal Q,\mathcal E})^\vee \ar@{_(->}[l] & (R^5_Q|x_0)^\vee \ar[u] & H^0(\Omega_C(2)) \ar@{->}[l] \ar[u]^{j_*}& K_{\mathcal E}=\mathbb C^2 \ar@{_(->}[l] \ar[u]\\
T^\vee_{(X,p)}\mathcal C_{\mathcal E} \ar[u]& T_{ J}^\vee \mathcal A_4^{} \ar@{->}[l] \ar[u]& (T_{J}\mathcal A_4^{}/T_X\mathcal C)^\vee \ar@{_(->}[l] \ar[u]& (R_F^2)^{\tau}  \ar[u]& H^0(\mathcal O_{\mathbb P^3}(2)) \ar@{->}[l]  \ar[u]& J^2_F\cap \mathbb C[x_0,\dots,x_3]_2=\mathbb C^3  \ar@{_(->}[l] \ar[u]
}
\end{equation}
with the vertical map on the right induced by the natural map from cubic threefolds $J^2_F\to J^4_Q$.  
\end{proposition}

\begin{proof}
For this we consider the following commutative diagram.
\begin{equation}
\xymatrix@R=.1em{
&& &\overline{\mathcal R}^a_{\mathcal Q} \ar@{-}[dd] \ar@{^(->}[rrrdd]& && \\
&& && && \\
\overline{\mathcal R}^a_{\mathcal Q,\mathcal E} \ar@{^(->}[rr] \ar@{->>}[ddd] \ar@{^(->}[rrruu]&&\overline{\mathcal R}^a_{3,4}\ar@{^(->}[rr] \ar@{->>}[ddd] &\ar@{->>}[d] & \overline {\mathcal R}_6^\tau \ar@{^(->}[rr] \ar@{->>}[ddd]&& \overline {\mathcal R}^a_6 \ar@{->>}[ddd]\\
&& &\mathcal C \ar@{^(->}[rrrdd]& &&\\
&& && && \\
\mathcal C_{\mathcal E}\ar[rr] \ar@{^(->}[rrruu] \ar@/_1pc/[rrrr]_{\ }&&\mathcal A^{(1,2,2,2)}_4 &&\mathcal A_{5}^\tau \ar[ll] \ar@{^(->}[rr]&& \mathcal A_5
}
\end{equation}
Here $\mathcal A_5^\tau$ is the moduli of principally polarized abelian varieties admitting an involution $\tau$ so that the invariant part has dimension $1$ and the anti-invariant part has dimension $4$.  By \cite[Cor.~ 5.4]{Rod_avgroup}, $\calA_5^\tau$ is an irreducible subvariety of $\calA_5$ of dimension $11$.  The map $\mathcal A^\tau_5\to \mathcal A_4^{(1,2,2,2)}$ is the map taking a principally polarized abelian variety to the dual of the anti-invariant part.  
We set $\overline{\mathcal R}_6^\tau:=\mathcal A_5^\tau\times_{\mathcal A_5}\overline{\mathcal R}_6^a$ to be the space of allowable covers admitting an involution $\tau$ so that the invariant part of the Prym variety  has dimension $1$ and the anti-invariant part has dimension $4$.   The map $\overline{\mathcal R}_{3,4}^a \to \overline {\mathcal R}_6^a$ is the 
 map that attaches a $\mathbb P^1$ at the branch points of the base curve, and takes the branched cover of this $\mathbb P^1$ at the attaching points and glues this to the cover curve at the branch points.   Such a cover has the extra involution $\tau$ by taking $\iota$ on the cover of the plane quartic, and the identity on the other component of the cover.  Thus the map just defined has image contained in $\overline{\mathcal R}^\tau_6$ as indicated in the diagram.   The map $\overline{\mathcal R}_{\mathcal Q,\mathcal E}\to \overline{\mathcal R}_{\mathcal Q}^a$ is given by attaching a line to the plane quartic at the marked points, and then taking the cover as indicated above.  
 
To analyze the diagram, we describe a few more differentials.
In the bottom row, given $(X,p)\in \mathcal C_{\mathcal E}$,  then we have
$$
\xymatrix@R=1em{
T_{(JX^{-\tau})^\vee} \mathcal A_4^{(1,2,2,2)} \ar@{=}[d] &T_{JX}\mathcal A_5^{\tau }=(T_{JX}\mathcal A_5)^\tau  \ar[l] \ar@^{->}[r] \ar@{=}[d]&T_{JX}\mathcal A_5 \ar@{=}[d]\\
 \operatorname{SymHom}((R^1_F)^{-\tau},(R^4_F)^{-\tau}) \ar@{=}[d]& \operatorname{SymHom}^\tau(R^1_F,R^4_F) \ar[l] \ar@^{->}[r] \ar@{=}[d]&\operatorname{SymHom}(R^1_F,R^4_F) \ar@{=}[d] \\
 (J^2_F\cap \mathbb C[x_0,\dots,x_3]_2)^\vee &((J^2_F)^\vee)^\tau  \ar[l] \ar@^{->}[r] & (J^2_F)^\vee
}
$$
where the map on the left is the one induced by the decomposition $R^i_F=(R^i_F)^\tau\oplus (R^i_F)^{-\tau}$. Note that on the left, we are also choosing coordinates for $F$ as in \S \ref{S:diffEck}.

In the top row, given $\pi_C:\widetilde C\to C$ in $\overline{\mathcal R}_{\mathcal Q,\mathcal E}^a$, with associated cover $\pi_D:\widetilde D\to D$ in $\overline{\mathcal R}^\tau_a$, then we have 
$$
\xymatrix@R=1em{
T_{\pi_C} \overline{\mathcal R}_{3,4}^a \ar[r] \ar@{=}[d] &T_{\pi_D}  \overline{\mathcal R}_6^\tau = (T_{\pi_D}\overline{\mathcal R}_6^a)^\tau \ar@^{->}[r] \ar@{=}[d]&T_{\pi_D}\overline{\mathcal R}_6^a \ar@{=}[d]\\
H^0(\Omega_C\otimes \omega_C(p_1+\cdots+p_4))^\vee  
\ar@{->}[d]_{j_*} \ar[r] & (H^0(\Omega_D\otimes \omega_D)^\vee)^\tau  \ar@^{->}[r] \ar@{->}[d]_{j_*}&H^0(\Omega_D\otimes \omega_D)^\vee \ar@{->}[d]_{j_*} \\
H^0(\omega_C^{\otimes 2}(p_1+\cdots+p_4))^\vee  
\ar@{=}[d] \ar[r] & (H^0(\omega_D^{\otimes 2})^\vee)^\tau  \ar@^{->}[r] \ar@{=}[d]&H^0(\omega_D^{\otimes 2})^\vee \ar@{=}[d] \\
 H^0(\mathcal O_C(3))^\vee \ar[r] &(H^0(\mathcal O_D(4))^\vee)^\tau  \ar@^{->}[r] & H^0(\mathcal O_D(4))^\vee
}
$$

Now tracing through the diagram and the given differentials, one obtains the stated result.
\end{proof}

Next we use Lemmas \ref{L:DSV15-explained} and \ref{L:DSV15},  and \propositionref{P:stablequartic} to give the following corollary:

\begin{corollary}\label{C:DSV15}
Let $(X,p)\subset \mathbb P\mathbb C^5 \cong \mathbb P^4$ be a smooth Eckardt cubic threefold, and let $\ell\subset X$ be a line through the Eckardt point $p$ giving rise to 
an odd pseudo-double cover $\pi:\widetilde C\cup \widetilde L\to C\cup L$ of a plane quintic $(Q=0)=C\cup L$ obtained as the union of a plane quartic $C$ and a transverse line $L$  in $\mathbb P\mathbb C^3 \cong \mathbb P^2$  (cf.~\propositionref{prop:projcone}).
Take coordinates on $\mathbb C^5$ as in \eqref{eqn:X'} so that $\ell=(x_0=x_1=x_2=0)$ and $X$ is defined by the equation $(F:=k x_3^2+2qx_3+c+lx_4^2=0)$.  Then $Q=(kc-q^2)l$, $C=(kc-q^2=0)$,  $L=(l=0)$ (see \eqref{eqn:CL}), and 
projection from $\ell$ gives the map  $\pi_\ell:\mathbb P\mathbb C^5\dashrightarrow  \mathbb P\mathbb C^3$, $[x_0,x_1,x_2,x_3,x_4]\mapsto [x_0,x_1,x_2]$. 
 
Under the identifications
$$
\xymatrix@R=.1em{
\mathbb C^3 \ar[r]^\sim& J^4_Q& \mathbb C^5\ar[r]^\sim& J^2_F\\
(a_0,a_1,a_2)\ar@{|->}[r]& \sum_{i=0}^2 a_i\frac{\partial Q}{\partial x_i}& (a_0,a_1,a_2,a_3,a_4)\ar@{|->}[r] &\sum_{i=0}^4 a_i\frac{\partial F}{\partial x_i}
} 
$$
the projectivization of the morphism $J^2_F\to J^4_Q$ of \eqref{E:CubicCoDiffDi} is the rational map $\pi_\ell$, and the projectivization of the dual map $(J^4_Q)^\vee\to (J^2_F)^\vee$ is therefore the inclusion $(\mathbb P^2_\ell)^\vee \subset (\mathbb P^4)^\vee$ corresponding to hyperplanes containing $\ell$ (cf.~\lemmaref{L:DSV15}).    

The map $J^2_F\cap \mathbb C[x_0,\dots,x_3]_2\to J^4_Q\cap l\mathbb C[x_0,x_1,x_2]_3$ in 
\eqref{E:EckCubicCoDiffDiAll} (note that in \eqref{E:EckCubicCoDiffDiAll} we have further simplified the notation by taking coordinates with $l=x_0$) is the restriction of the map $J^2_F\to J^4_Q$ defined above. Therefore, the projectivization of the dual map is the inclusion of the subset of $(\mathbb P^2_\ell)^\vee$ corresponding to hyperplanes $\sum_{i=0}^2 a_ix_i=0$ containing $\ell$ (and hence containing the Eckardt point $p \in \ell$) and such that $l(a_0,a_1,a_2)=0$ (compare \eqref{E:CokerQuarLine} and see also \propositionref{P:stablequartic}) into the subset of $(\mathbb P^4)^\vee$ parameterzing hyperplanes $\sum_{i=0}^3 a_ix_i=0$ containing the Eckardt point $p$ and satisfying $l(a_0,a_1,a_2)=0$ (compare \eqref{E:CokerEckCub}).
\end{corollary}

\begin{proof}
 This follows immediately from Lemmas \ref{L:DSV15-explained} and \ref{L:DSV15},  and \propositionref{P:stablequartic}.
 \end{proof}

\subsection{Proof of \theoremref{thm:prymfiber}}

\begin{proof}[Proof of \theoremref{thm:prymfiber}]
We apply the strategy of \eqref{E:FultonSegre} and \lemmaref{L:DS-fiber} to the diagram
\eqref{E:EckFiberMainDi}, 
where the image of $\mathcal C_{\mathcal E}$ plays the role of $Z$, and the image of $\overline {\mathcal R}_{\mathcal Q, \mathcal E}^a$ plays the role of $Z_0'$.

\corollaryref{C:DSV15} implies that the differential of the Prym map  $\overline{\mathcal R}^a_{3,4}\to \mathcal A_4^{(1,2,2,2)}$ at a cover $(\pi:\widetilde C\to C)\in \overline {\mathcal R}_{\mathcal Q,\mathcal E}^a$ has kernel exactly of dimension $1$.  Thus at $\pi$, the fiber of $\mathcal P_{3,4}$ over $P(\widetilde C,C)=(JX^{-\tau})^\vee$ is exactly $E$, the curve in the Fano surface of lines $F(X)$ consisting of lines through the Eckardt point $p$.  Since this holds at every cover $\pi$, this implies that the image of $\overline {\mathcal R}_{\mathcal Q,\mathcal E}^a$ is an irreducible and connected component of the fiber over the image of $\mathcal C_{\mathcal E}$, with generic fiber of the image of $\overline {\mathcal R}_{\mathcal Q,\mathcal E}^a$ over the image of  $\mathcal C_{\mathcal E}$ connected. 

  At the same time, let us use  \corollaryref{C:DSV15} to compute the degree of the map of exceptional divisors $\widetilde{\overline{\mathcal R}_{\mathcal Q,\mathcal E}^a}\to \widetilde {\mathcal C}_{\mathcal E}$ after blowing up along the corresponding loci (see \S \ref{S:DS-method}).  For this, we can pick a general Eckardt cubic threefold $(X,p)$, and a general point in the projectivization of the fiber of the normal bundle of the image at the image of $(X,p)$, corresponding by \corollaryref{C:DSV15} to a general hyperplane $H\subset \mathbb P^4$ passing through the Eckardt point.  The lemma says that the points lying above this chosen point correspond to those lines  $\ell'\subset X$ that pass through the Eckardt point $p$ and lie in the hyperplane section $H$.  Since the lines through $p$ on $X$ are parameterized by $E\subset S\subset X$, and $E\cap H$ consists of three points, there are three such lines.

Since the degree of the Prym map is $3$, we can conclude from \lemmaref{L:DS-fiber} that 
$\overline {\mathcal R}_{\mathcal Q,\mathcal E}^a=\mathcal C_{\mathcal E}\times_{\mathcal A_4^{(1,2,2,2)}}\overline{\mathcal R}_{3,4}^a$, and we are done.
\end{proof}

\begin{proof}[Proof of \corollaryref{C:DSdegree}]
This now follows immediately from \eqref{E:degComp}, and the proof of \theoremref{thm:DSprymfiber}.  
\end{proof}

\section{Global Torelli theorem for Eckardt cubic threefolds}\label{S:GlobaTor}

Applying \theoremref{thm:prymfiber}, we prove the global Torelli theorem for the period map $\calP: \calM \rightarrow \calA_4^{(1,1,1,2)}$ of cubic surface pairs defined in Section \ref{sec:eckardtcubic3}. 

\begin{theorem}[Global Torelli] \label{thm:cubic2pairtorelli}
The period map $\calP: \calM \rightarrow \calA_4^{(1,1,1,2)}$ for cubic surface pairs is injective.
\end{theorem}
\begin{proof}
Using the isomorphism $\mathcal M\cong \mathcal C_{\mathcal E}$, we prove that the period map $\mathcal C_{\mathcal E}\to \mathcal A_4^{(1,1,1,2)}$ is injective.  
In other words, let $(X_0, p_0)$ and $(X_1, p_1)$ be Eckardt cubic threefolds coming from cubic surface pairs $(S_i, \Pi_i)$. Denoting the involution associated with $p_i$ by $\tau_i$ for $i=0,1$, we will  prove that if the anti-invariant parts are isomorphic to each other $JX_0^{-\tau_0} \cong JX_1^{-\tau_1}$, then $(X_0,p_0)$ is isomorphic to $(X_1,p_1)$.

 First recall that from \theoremref{T:conedecomp}, for $i=0,1$, there exist a $\tau_i$-invariant line $\ell'_i \subset X_i$ through $p_i$, such that when projecting $X_i$ from $\ell'_i$ we get   a covering  $\pi_i:\widetilde{C}_i \rightarrow C_i$   in $\calR_{3,4}$ such that $P(\widetilde C_i,C_i)\cong (JX_i^{-\tau_i})^\vee$ as polarized abelian varieties.   More precisely,  the double cover $\widetilde{C}_i \rightarrow C_i$  is obtained by projecting $X_i$ from $\ell'_i$ and restricting the discriminant pseudo-double cover $\pi_i: \widetilde{C}_i \cup E_i \rightarrow C_i \cup L_i$   to the smooth quartic component $C_i$.  
Recall that $\widetilde{C}_i \rightarrow C_i$ ($i=0,1$)  and $E_i\to L_i$ are branched at the four intersection points $C_i \cap L_i$, and therefore the branched double cover $\widetilde{C}_i \rightarrow C_i$ determines the entire discriminant pseudo-double cover.   Recall also that the cover $\pi_i:\widetilde C_i\to C_i$ is determined by an odd theta characteristic  (i.e., bi-tangent) $\kappa_{C_i}$ on $C_i$ (\lemmaref{L:kappaCEck}), and that this theta characteristic also determines the cubic surface $S_i$ associated to $(X_i,p_i)$  (see \eqref{E:kappCEck} and \S \ref{S:CubSurfPoint}).  The plane $\Pi_i$ is  determined by $L_i$; it is the pre-image of $L_i$ under the projection of $\mathbb P^3=(x_4=0)$ from $p_i'=\ell_i'\cap S_i'$.    In particular, the data   $(X_i,p_i,\ell_i)$, up to projective linear transformations, is equivalent to the data  $(C_i,\kappa_{C_i},L_i)$ (see \theoremref{T:Eck-from-quart}).

Now since $P(\widetilde C_1,C_1)\cong P(\widetilde C_0,C_0)$, we have from \theoremref{thm:prymfiber} that there exists 
a $\tau_0$-invariant line $\ell''_0 \subset X_0$ through $p_0$, such that when projecting $X_0$ from $\ell''_0$ we get  the  covering  $\pi_1:\widetilde{C}_1 \rightarrow C_1$.  In fact, as described above, projecting from $\ell''_0$, we get the full discriminant $\widetilde C_1\cup E_1\rightarrow C_1\cup L_1$.  
But then the triple $(X_1,p_1,\ell_1')$ is projectively equivalent to $(X_0,p_0,\ell_0'')$.  In other words, $(X_1,p_1)\cong (X_0,p_0)$.

\end{proof}

\section{Eckardt cubic threefolds as fibrations in conics 2: point-wise invariant lines} \label{sec:proj27}

Let $X$ be a cubic threefold with an Eckardt point $p$ (see Equation (\ref{eqn:eckardt})). Denote the associated involution in (\ref{eqn:tau}) by $\tau$. In Section \ref{sec:proj27}, we study the $\tau$-decomposition of $JX$ (under a genericity assumption; see \cndref{cond:generic})
  via the linear projection of $X$ from a line point-wise fixed by $\tau$ (i.e., from one of the $27$ lines on the cubic surface $S\subset X$). As an application, we prove that the period map $\calP$ for cubic surface pairs (defined in Section \ref{sec:eckardtcubic3}) is generically finite onto its image. 
 As the results in this section are not used in the proofs of our main theorems, and the proofs of the result in this section are similar to those is \S\ref{sec:projcone}, we give more brief treatment of the proofs in this section.

\subsection{Projecting Eckardt cubic threefolds from point-wise invariant lines} We now revisit \S \ref{S:CubSolid} in the case of an Eckardt cubic threefold and a point-wise fixed line.  
More precisely, let $(X,p)$ be an Eckardt cubic threefold and let $\ell \subset S$ be a line contained in the associated cubic surface \eqref{E:Sdef}.
Let us choose coordinates so that  $X$ is given by Equation (\ref{eqn:eckardt}) 
$f(x_0, \dots, x_3) + l(x_0,\dots,x_3)x_4^2 = 0$, 
with the Eckardt point $$p=[0,0,0,0,1].$$
  Let $L_1,L_2,L_3$ be linear forms on $\mathbb P^4$ with $\ell=(L_1=L_2=L_3=0)$.  Since $\ell\subset S=X\cap (x_4=0)=0$, we can assume that $L_1=x_4$.  As a consequence, we can assume that $L_2,L_3$ do not include $x_4$.  Therefore, by a change of coordinates in $x_0,x_1,x_2,x_3$, we can assume that 
 the line $\ell$ is cut out by $x_2=x_3=x_4=0$:
 $$
 \ell=(x_2=x_3=x_4=0).
 $$
Because $\ell$ is not contained in the cone $X \cap T_pX$, the linear polynomial $l$ in Equation (\ref{eqn:eckardt}) contains either $x_0$ or $x_1$ (note that $T_pX=(l=0)$). Interchanging $x_0$ and $x_1$, we assume it is $x_0$. After a change of coordinates (namely, $x_0 \mapsto l$ and $x_i \mapsto x_i$ for $1 \leq i \leq 4$), the Eckardt cubic threefold $X$ is given by 
\begin{equation} \label{eqn:eckst}
f(x_0, \dots, x_3) + x_0x_4^2 = 0.
\end{equation}
Then we have 
$$
S=(f(x_0,\dots,x_3)=0)\subset \mathbb P^3_{x_0,x_1,x_2,x_3}
$$
and the hyperplane section is
\begin{equation}
E=(x_0=0)\cap S=(x_0=f=0).
\end{equation}

Now we project $X$ from the line $\ell$ to a complementary plane $\bP^2_{x_2,x_3,x_4}=(x_0=x_1=0)$ and obtain a conic bundle $\pi_\ell: \operatorname{Bl}_\ell X \rightarrow \bP^2_{x_2,x_3,x_4}$.  Following \eqref{E:CubSolid}, we write  the equation of $X$ as
\begin{equation} \label{eqn:X}
l_1(x_2,x_3)x_0^2+2l_2(x_2,x_3)x_0x_1+l_3(x_2,x_3)x_1^2+\end{equation}
$$
2(q_1(x_2,x_3)+\frac12x_4^2)x_0+2q_2(x_2,x_3)x_1+c(x_2,x_3)=0
$$
where $l_i$, $q_j$ and $c$ are homogeneous polynomials in $x_2,x_3$ of degree $1$, $2$ and $3$ respectively. (In \eqref{E:CubSolid} these were polynomials in $x_2,x_3,x_4$, but here we have $x_4$ appearing in (\ref{eqn:eckst}) only in the monomial $x_0x_4^2$ and so we have modified the notation slightly to reflect this.) 

We now consider the associated matrix \eqref{E:Msolid}:
\begin{equation} \label{eqn:A}
M=\begin{pmatrix}  
l_1 &l_2 &q_1+\frac12x_4^2\\
l_2&l_3&q_2\\
q_1+\frac12x_4^2&q_2&c\\
\end{pmatrix}.
\end{equation}

The equation of the discriminant $D\subset \mathbb P^2$ is \eqref{E:Deqsolid}: 
\begin{equation} \label{eqn:D}
\det(M) = -\frac14l_3 x_4^4 +\det  \begin{pmatrix} l_2&q_1\\l_3&q_2\end{pmatrix}x_4^2 + \det \begin{pmatrix}  l_1 &l_2 &q_1\\
l_2&l_3&q_2\\
q_1&q_2&c\\
\end{pmatrix} = 0.
\end{equation}

Note that in our coordinates, the involution $\tau$ on $X$ is induced by an involution $\tau$ on $\mathbb P^4$, namely $x_4\mapsto -x_4$.  This induces an involution $\tau$ on $\mathbb P^2_{x_2,x_3,x_4}$ given also by $x_4\mapsto -x_4$.  Since $\ell$ is fixed by $\tau$, there is an induced involution on $\operatorname{Bl}_\ell X$, and from the definitions it is clear that $\pi_\ell:\operatorname{Bl}_\ell X\to \mathbb P^2_{x_2,x_3,x_4}$ is equivariant with respect to the involution $\tau$.  Restricting to the discriminant $D$, we see that $\tau$ induces  involutions
\begin{equation}\label{E:sigmaDef}
\sigma:\widetilde D\to \widetilde D \ \ \text{ and } \ \ \ \sigma_D:D\to D
\end{equation}
 making the cover $\widetilde D\to D$ equivariant.  

\begin{rem}
 Fiberwise, we can describe the involutions $\sigma$ and $\sigma_D$ as follows.  The involution  $\tau$ on $X$ sends a degenerate fiber $\ell \cup m \cup m'$  to another degenerate fiber $\ell \cup \tau(m) \cup \tau(m')$; this then defines the action of $\sigma$ on $\widetilde D$. Let $\iota: \widetilde{D} \rightarrow \widetilde{D}$ be the covering involution associated with $\widetilde{D} \rightarrow D$. From the previous geometric description of $\sigma$, we deduce that $\sigma \iota = \iota \sigma$. Then $\sigma$ induces an involution on $D$ and one verifies easily that this involution coincides with $\sigma_D$.
\end{rem}

\begin{lemma} \label{lem:Dbar}
If the discriminant curve $D$ is smooth, then the quotient curve $\overline{D} := D/\sigma_D$ is smooth of genus $2$. 
\end{lemma}

\begin{proof}
For brevity, the proof is left to the reader.
\end{proof}

\begin{proposition}[Projecting from $\ell\subset S\subset X$] \label{prop:proj27}
Let $(X,p)$ be an Eckardt cubic threefold and let $\ell \subset S$ be a line contained in the associated cubic surface \eqref{E:Sdef}.
The following are equivalent:
\begin{enumerate}
\item  $(X,p)$ and $\ell$ satisfy \cndref{cond:generic} (note that such a line $\ell\subset S$ exists on a general Eckardt cubic threefold; see Remark \ref{R:condGen}).
\item The discriminant plane quintic 
  $D$ is smooth and the double cover $\widetilde{D} \rightarrow D$ is connected and \'etale.  
\end{enumerate}
\end{proposition}

\begin{proof}
For brevity, this is left to the reader.
\end{proof}

\subsection{Klein group towers of coverings}
We are interested in studying the Prym variety $P(\widetilde D,D)$.  
The key point is that  the covering curve $\widetilde D$ admits two commuting involutions, namely $\sigma$ induced from $\tau$ \eqref{E:sigmaDef}, and $\iota$ induced from the double cover $\widetilde D\to D$.   
It has been clear going back to \cite{Mumford_prym} that one should consider the associated tower of covers \eqref{E:kleinD}  induced by taking quotients of $\widetilde D$ by the various subgroups of  $\langle \sigma,\iota \rangle \subset \operatorname{Aut}(\widetilde D)$.  While Mumford focused on a particular case involving hyperelliptic curves, this general approach, including studying more complicated automorphism groups of the covering curve of a branched double cover, was explored in more depth in  \cite{Donagi_prymfiber},  and  then generalized in \cite{RR_prym1} to include the case we study here.  We explain this in the context of double covers of discriminant curves of Eckardt cubic threefolds.

We introduce the following notation. For any element $g \neq 1$ of the Klein four group $\langle \sigma, \iota \rangle \subset \mathrm{Aut}(\widetilde{D})$, we denote the quotient curve 
 $$
 \widetilde{D}_g:=\widetilde{D}/\langle g\rangle.
 $$
In particular, $\widetilde{D}_\iota=D$.

\begin{lemma} \label{lem:kleinD}
We have the following commutative diagram:
\begin{equation}\label{E:kleinD}
\xymatrix{
& \widetilde{D} \ar[ld]_{a_\sigma} \ar[d]^{a_{\sigma \iota}}_{\text{\'et}} \ar[rd]^{a_\iota}_{\text{\'et}} \\
\widetilde{D}_\sigma \ar[rd]_{b_\sigma}^{\text{\'et}} &
\widetilde{D}_{\sigma \iota} \ar[d]^{b_{\sigma \iota}} &
\widetilde{D}_\iota = D \ar[ld]^{b_\iota} \\
& \overline{D} 
}
\end{equation}
Moreover, 
\begin{enumerate}
\item The map $a_\sigma$ is a double covering map branched at twelve points. The maps $a_{\sigma \iota}$ and $a_\iota$ are both \'etale double covering maps.
\item The map $b_\sigma$ is an \'etale double covering map. Both $b_{\sigma \iota}$ and $b_\iota$ are double covering maps ramified at six points.
\item The curves are all smooth and their genera are given as follows: $g(\widetilde{D}) = 11$, $g(\widetilde{D}_\sigma)=3$, $g(\widetilde{D}_{\sigma \iota})=6$, $g(D)=6$ and $g(\overline{D})=2$. 
\end{enumerate}
\end{lemma}

\begin{proof}
This essentially follows from \cite[Thm.~ 6.3]{RR_prym1}.  For brevity, the details are left to the reader.
\end{proof}

\begin{proposition}[{\cite[Thm.~ 6.3]{RR_prym1}}]\label{P:RR6.3} In the notation of \eqref{E:kleinD}, let $(P(\widetilde D,D),\Xi)$ be the principally polarized Prym variety.
There is an isogeny of polarized abelian varieties 
$$
\phi_\iota: P(\widetilde{D}_\sigma, \overline{D}) \times P(\widetilde{D}_{\sigma \iota}, \overline{D}) \longrightarrow P(\widetilde{D}, D), \,\,\,\,\,\, (y_1, y_2) \mapsto a_\sigma^*(y_1) + a_{\sigma \iota}^*(y_2)
$$
with $\ker(\phi_\iota) \cong (\bZ/2\bZ)^3$,
where $a_\sigma^*$ denotes the pull-back between Jacobians, and similarly for $a_{\sigma\iota}^*$.  
 More explicitly, we have $$
\ker(\phi_\iota) = \{(y_1, y_2) \in P(\widetilde{D}_\sigma, \overline{D})[2] \times P(\widetilde{D}_{\sigma \iota}, \overline{D})[2] \,\mid \, a_\sigma^*(y_1) = a_{\sigma \iota}^*(y_2)\}.
$$ 

Moreover, with respect to the action of $\sigma$ on $(P(\widetilde D,D),\Xi)$, the isogeny $\phi_\iota$ induces  isomorphisms of polarized abelian varieties  $P(\widetilde D,D)^\sigma\cong P(\widetilde D_\sigma,\overline D)$ and $P(\widetilde D,D)^{-\sigma}\cong P(\widetilde D_{\sigma \iota},\overline D)/\langle  b_{\sigma \iota}^*\epsilon \rangle $, where $\epsilon$ is the $2$-torsion line bundle on $\overline D$  defining the \'etale double cover $\widetilde{D}_\sigma \rightarrow \overline{D}$, and $b_{\sigma \iota}^*\epsilon$ is nontrivial.
\end{proposition}

\begin{proof}
This essentially follows from \cite[Thm.~ 6.3]{RR_prym1}.  For brevity, the details are left to the reader.
\end{proof}

Another Klein tower shows up in Mumford's hyperelliptic construction for the \'etale double cover  $b_\sigma: \widetilde D_\sigma \to \overline D$. 

\begin{lem}[{\cite[p.346]{Mumford_prym}}]\label{L:Mum} In the notation of \lemmaref{lem:kleinD}, with $E$ the hyperplane section of the cubic surface $S$ \eqref{E:Edef},  
we have the following commutative diagram
\begin{equation}\label{E:MumPrym}
\xymatrix{
& \widetilde{D}_\sigma \ar[ld]_{b_\sigma=b_{\overline{D}}}^{\text{\'et}} \ar[d]^{b_E} \ar[rd]^{b_R} \\
\overline{D} \ar[rd]_{c_{\overline{D}}} &
E \ar[d]^{c_E} &
R \cong \bP^1 \ar[ld]^{c_R} \\
& T \cong \bP^1
}
\end{equation}
and we have an isomorphism of principally polarized abelian varieties
 $$P(\widetilde D_\sigma,\overline D)\cong E\times J(R)=E.$$
\end{lem}

\begin{proof}
This can be deduced from \cite[Prop.~ 4.2, Cor.~ 4.3]{BO_pas14} (see also \cite[\S 7]{Mumford_prym}).  For brevity, the details are left to the reader.
\end{proof}

Putting this together, we obtain the following theorem:

\begin{theorem}[Projecting from $\ell\subset S\subset X$]  \label{T:27decomp}
Let $(X,p)$ be an Eckardt cubic threefold, let $\ell \subset S$ be a line contained in the associated cubic surface \eqref{E:Sdef} 
 satisfying \cndref{cond:generic} (note that such a line $\ell\subset S$ exists on a general Eckardt cubic threefold; see Remark \ref{R:condGen}), 
let $E\subset S$ be the hyperplane section  \eqref{E:Edef}, and let $b_{\sigma\iota}:\widetilde D_{\sigma\iota}\to \overline D$ be the branched cover of the smooth genus $2$ curve $\overline D$ from \eqref{E:kleinD}.  

There is an isogeny of polarized abelian varieties 
$$
\phi: E \times P(\widetilde{D}_{\sigma \iota}, \overline{D}) \longrightarrow JX
$$
with $\ker(\phi) \cong (\bZ/2\bZ)^3$.

Moreover, with respect to the action of $\tau$ on $(JX,\Theta_X)$, the isogeny $\phi$ induces  isomorphisms of polarized abelian varieties  $JX^\tau \cong E$ and $JX^{-\tau}\cong P(\widetilde D_{\sigma \iota},\overline D)/\langle  b_{\sigma \iota}^*\epsilon \rangle $, where $\epsilon$ is the $2$-torsion line bundle defining the \'etale double cover $\widetilde{D}_\sigma \rightarrow \overline{D}$ from \eqref{E:kleinD}, and $b_{\sigma \iota}^*\epsilon$ is nontrivial.
\end{theorem}

\begin{proof}
Since $JX\cong P(\widetilde D,D)$, and the action of $\tau$ on $JX$ is identified with the action of $\sigma$ on $P(\widetilde D,D)$, 
this is an immediate consequence of \propositionref{P:RR6.3} and \lemmaref{L:Mum}.
\end{proof}

\begin{rem}
Using results of
\cite{NO_prymtorelli2, NO_prymtorelli, LO_cyclicprym},  one can use \theoremref{T:27decomp} to prove that the period map $\calP: \calM \rightarrow \calA_4^{(1,1,1,2)}$ for cubic surface pairs is generically finite-to-one onto its image, and that the differential of $\calP$ is injective at a generic point.   Since we can use other methods to give a short proof of a stronger statement  (\corollaryref{C:InfTorHdg}), we omit this proof here.
\end{rem}

\bibliography{ref}
\end{document}